\documentclass{aims}
\usepackage{amsmath}
\usepackage{paralist}
\usepackage{amssymb}  
\usepackage{epic}
\usepackage{times}

  \usepackage{graphics} 
  \usepackage{epsfig} 
 \usepackage[colorlinks=true]{hyperref}
\hypersetup{urlcolor=blue, citecolor=red}

\def\currentvolume{X}
 \def\currentissue{X}
  \def\currentyear{200X}
   \def\currentmonth{XX}
    \def\ppages{X--XX}

\newtheorem{theorem}{Theorem}[section]

\newtheorem{lemma}[theorem]{Lemma}

\theoremstyle{definition}

\newtheorem{remark}{Remark}

\def\wh{\widehat}

\def\fS{{\mathfrak S}}
\def\<{\leqslant}           
\def\>{\geqslant}           
\def\d{\partial}

\def\x{\times}

\def\mZ{{\mathbb Z}}

\def\var{{\bf var}}
\def\bE{{\bf E}}
\def\bM{{\bf M}}

\def\bm{{\bf m}}

\def\bone{{\bf 1}}

\def\re{{\rm e}}

\def\bI{{\bf I}}

\def\mH{{\mathbb H}}

\def\cP{{\mathcal P}}

\def\sP{{\sf P}}

\def\cov{{\bf cov}}

\def\bra{{\langle}}
\def\ket{{\rangle}}
\def\Bra{\left\langle}
\def\Ket{\right\rangle}

\def\cW{{\mathcal W}}
\def\cN{{\mathcal N}}

\def\mR{{\mathbb R}}

\def\mZ{{\mathbb Z}}

\def\rd{{\rm d}}

\def\phi{\varphi}

\def\rT{{\rm T}}

\def\sn{{| \! | \! |}}
\def\[[{{[ \! [}}
\def\]]{{] \! ]}}

\def\[[[{{[ \! [ \! [}}
\def\]]]{{] \! ] \! ]}}

\title[Moment Lyapunov exponents of Gaussian random fields]
      {The monomer-dimer problem and moment Lyapunov exponents of homogeneous
Gaussian random fields}

\author[Igor G. Vladimirov]{}
 \email{igor.g.vladimirov@gmail.com}
\begin{document}
\maketitle
\centerline{\scshape Igor G. Vladimirov}
\medskip
{\footnotesize
 \centerline{University of New South Wales}
   \centerline{Canberra, ACT 2600, Australia}
} 

\bigskip

 \centerline{Dedicated to the memory of Alexei Pokrovskii}

\begin{abstract}
We consider an ``elastic'' version of the statistical mechanical mo\-no\-mer-dimer problem on the $n$-dimensional integer lattice. Our setting includes the classical ``rigid'' formulation as a special case and extends it by allowing each dimer to consist of particles at arbitrarily distant  sites of the lattice, with the energy of interaction between the particles in a dimer depending on their relative position. We reduce the free energy of the elastic dimer-monomer  (EDM) system per lattice site in the thermodynamic limit to the moment Lyapunov exponent (MLE) of a homogeneous Gaussian random field (GRF) whose mean value and covariance function are the Boltzmann factors associated with the monomer energy and dimer potential. In particular, the classical monomer-dimer problem becomes related to the MLE of a moving average GRF. We outline an approach to recursive computation of the partition function for ``Manhattan'' EDM systems where the dimer potential is a weighted  $\ell_1$-distance and the auxiliary  GRF is a Markov random field of Pickard type which behaves  in space like autoregressive processes do in time. For one-dimensional Manhattan EDM systems, we compute  the MLE of the resulting Gaussian Markov chain as the largest eigenvalue of a compact transfer operator on a Hilbert space which is related to the annihilation and creation operators of the quantum harmonic oscillator and also recast it as the  eigenvalue problem for a pantograph functional-differential equation.
\end{abstract}

\section{Introduction}

It is not uncommon that deterministic problems may reveal connections with stochastic counterparts, and it is particularly so if the original formulation inherently involves a pseudo-randomness production mechanism or rich combinatorics.
%
When this happens, probability theoretic methods can provide an additional insight into the deterministic setting.
Such are, for example, the problems concerned with chaotic dynamical systems under increasingly refined spatial discretization which models the finite computer arithmetic.
 Here, the effect of the round-off errors, magnified in the long run by positive Lyapunov  exponents, leads to a subtle interplay  between the ergodic behaviour  of the underlying continuous system and the computer discretization. This can distort the invariant measure to the extent that, in place of its original absolute continuity,  a persistent attracting centre emerges, thus destroying resemblance between  the long-term behaviour of the discretised system and its continuous predecessor.
As reported  in a series of papers by A.V.Pokrovskii, V.S.Kozyakin and their collaborators (see, for example,  \cite{PKM_2000} and references therein), the attracting centre can be modelled through a  finite Markov chain with an absorbing state. Moreover, a strikingly accurate numerical reproduction of statistics of discretised systems has been achieved with the subsidiary Markov chains obtained  by composing random maps with a common fixed point on the finite state space,  sampled independently according to a specially designed probability measure.\footnote{Note in passing that the randomizing effect of the spatial discretization was clarified in rigorous terms by the author of the present paper in \cite{V_1997} by  showing that the round-off errors are independent and uniformly distributed (that is, form a non-Gaussian white noise-like sequence) with respect to a natural finitely additive probability measure (a frequency functional) on an algebra of quasiperiodic sets,  provided the original dynamical system is nonresonant;  see also \cite{KKPV_1997,VKD_2000}.}
The idea of random  composite maps on  a countable set can be extended to the continuum, for example,  by considering the compositions $\phi_1\circ \ldots \circ \phi_k$  of jointly Gaussian real-valued random processes $\phi_k(x)$ of a real parameter $x$. If these processes have differentiable sample paths and share a nonrandom fixed point $x_*$,   then its stability under the maps $\phi_k$, as a random dynamical system, is determined (in the linear approximation) by the asymptotic behaviour of the products
\begin{equation}
\label{psi}
    \psi_N
    :=
    \prod_{k=1}^{N}
    \xi_k
\end{equation}
of jointly Gaussian random variables $\xi_k:=\phi_k'(x_*)$ as $N\to +\infty$. Assuming that the processes $\phi_k$ are all identically distributed, the standard ergodicity argument under a suitable mixing condition  (on the decay of correlations between the processes $\phi_k$ with distant values of $k$) leads to the Lyapunov exponent
\begin{equation}
\label{LE}
    \lim_{N\to +\infty}
    \frac{\ln |\psi_N|}{N}
    =
    \bE \ln |\xi_1|.
\end{equation}
Here, the expectation $\bE(\cdot)$ involves only  the common probability law of the processes  $\phi_k$ as if the latter were independent copies of a ``mother'' process, in which case (\ref{LE}) would follow immediately from the strong law of large numbers \cite{Sh_1995}. Similar products of appropriately rescaled random processes have attracted an active research interest due to their applicability to multifractal modelling of turbulence cascades and other complex phenomena; see, for example, \cite{ALS_2008}.
The problem of investigating the asymptotic behaviour of the products $\psi_N$ from (\ref{psi}) becomes considerably harder if it is concerned with the \emph{moment Lyapunov exponent} (MLE) which we define here as
\begin{equation}
\label{MLEintro}
    \lim_{N\to +\infty}
    \frac{\ln |\bE \psi_N|}{N},
\end{equation}
provided this limit exists (otherwise, the upper limit or other limit points  can be considered instead).  In contrast to the conventional Lyapunov exponent (\ref{LE}), the MLE (\ref{MLEintro}) is easy to compute \emph{only}  in the case of independent random variables $\xi_k$ in (\ref{psi}). If they are dependent, their joint probability law  nontrivially enters the MLE even if correlations in the Gaussian random sequence $\xi_1, \xi_2, \xi_3, \ldots$ decay fast enough to ensure mixing. Although the problem of computing the MLE as a modification to the usual Lyapunov exponent (\ref{LE}) may seem artificial in the context of random dynamics stability mentioned above, it has a remarkable connection with a statistical mechanical problem of computing the equilibrium free energy of a system of monomers and dimers on a lattice, and this connection is the main theme of the present paper.

In the classical \emph{monomer-dimer} problem \cite{B_1982,E_1984,F_1961,FT_1961,HM_1970,K_1961,L_1967}, which is essentially solved in the planar case and remains unsolved in higher dimensions, the monomers are represented by singletons, and the dimers are two-element subsets of the lattice formed by nearest neighbours, with the energy of a dimer depending on its orientation. In particular, in a monomer-free ``isotropic'' setting on the plane, the partition function counts the number of ``domino'' tilings  of a  given finite subset of the planar lattice \cite[p.~125]{B_1982}. It turns out that, for any dimension $n$, the number of such partitions of a finite subset $\Lambda \subset \mZ^n$ of the $n$-dimensional integer lattice coincides with the product moment
\begin{equation}
\label{bM}
    \bM_{\Lambda}(\xi)
    :=
    \bE \prod_{x\in \Lambda}
    \xi_x
\end{equation}
of a specially constructed homogeneous Gaussian random field (GRF) $\xi:=(\xi_x)_{x\in \mZ^n}$. Moreover, this GRF is of moving average type which is easy to simulate. This, in principle, provides an alternative probabilistic way (different from the probabilistic algorithm proposed in \cite{KRS_1996}) for approximate computation of the partition function in the monomer-dimer problem by modelling the auxiliary GRF and statistically estimating its product moments instead of the direct Monte-Carlo simulation of the underlying particle system which is complicated by the tilability issue. Due to this connection, the free energy of the monomer-dimer system per lattice site in the thermodynamic limit is proportional to the MLE of the auxiliary GRF, defined, similarly to (\ref{MLEintro}), by
\begin{equation}
\label{bm}
    \bm(\xi)
    :=
    \lim_{\Lambda \to \infty}
    \frac{\ln |\bM_{\Lambda}(\xi)|}{\#\Lambda},
\end{equation}
where $\#(\cdot)$ is the number of elements in a finite set, and $\Lambda\to \infty$ is understood in the sense of van Hove \cite{R_1978}. This connection between the MLEs of homogeneous GRFs and the statistical mechanical partition functions remains valid for a wider class of monomer-dimer systems.

In the present paper, we consider an ``elastic'' version of the monomer-dimer problem, which extends the original ``rigid'' formulation by allowing each dimer to consist of two particles at arbitrarily distant sites of the lattice, with the energy of interaction between the particles depending on their relative position. Thus, the geometric constraint that the dimers are formed only by nearest neighbours is removed, thereby introducing the long-range interactions; see Fig.~\ref{fig:edm_confs}.
\begin{figure}[htbp]
    \vskip-2cm
\centering
\includegraphics[width=5cm]{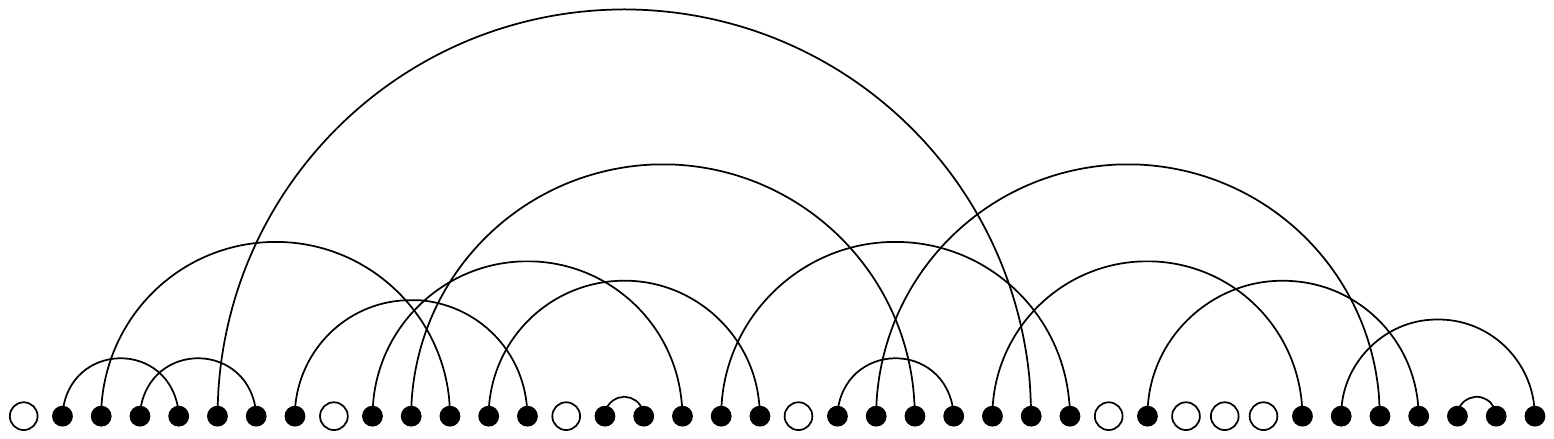}\hskip-1.3cm
\includegraphics[width=5cm]{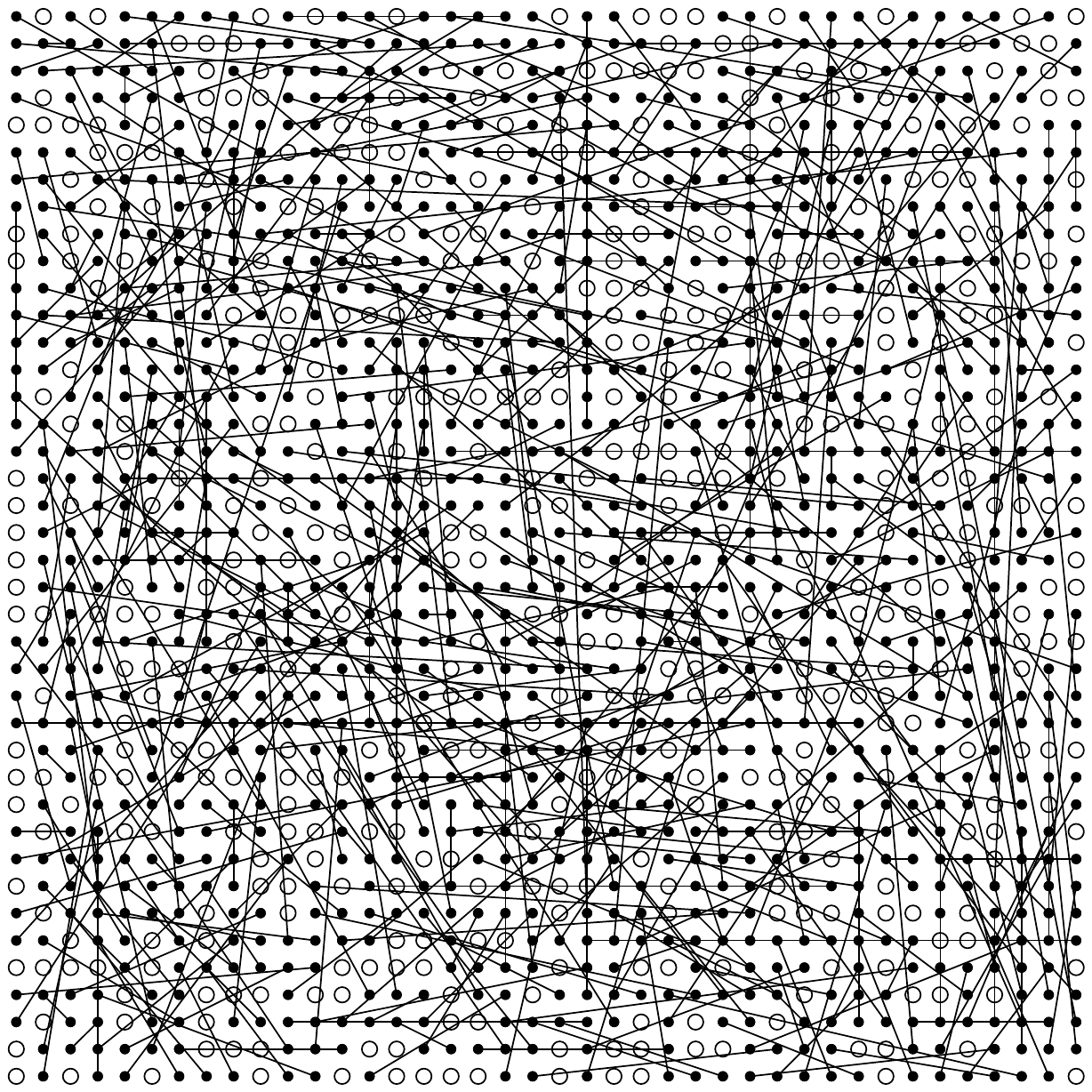}\hskip-2cm
\includegraphics[width=60mm]{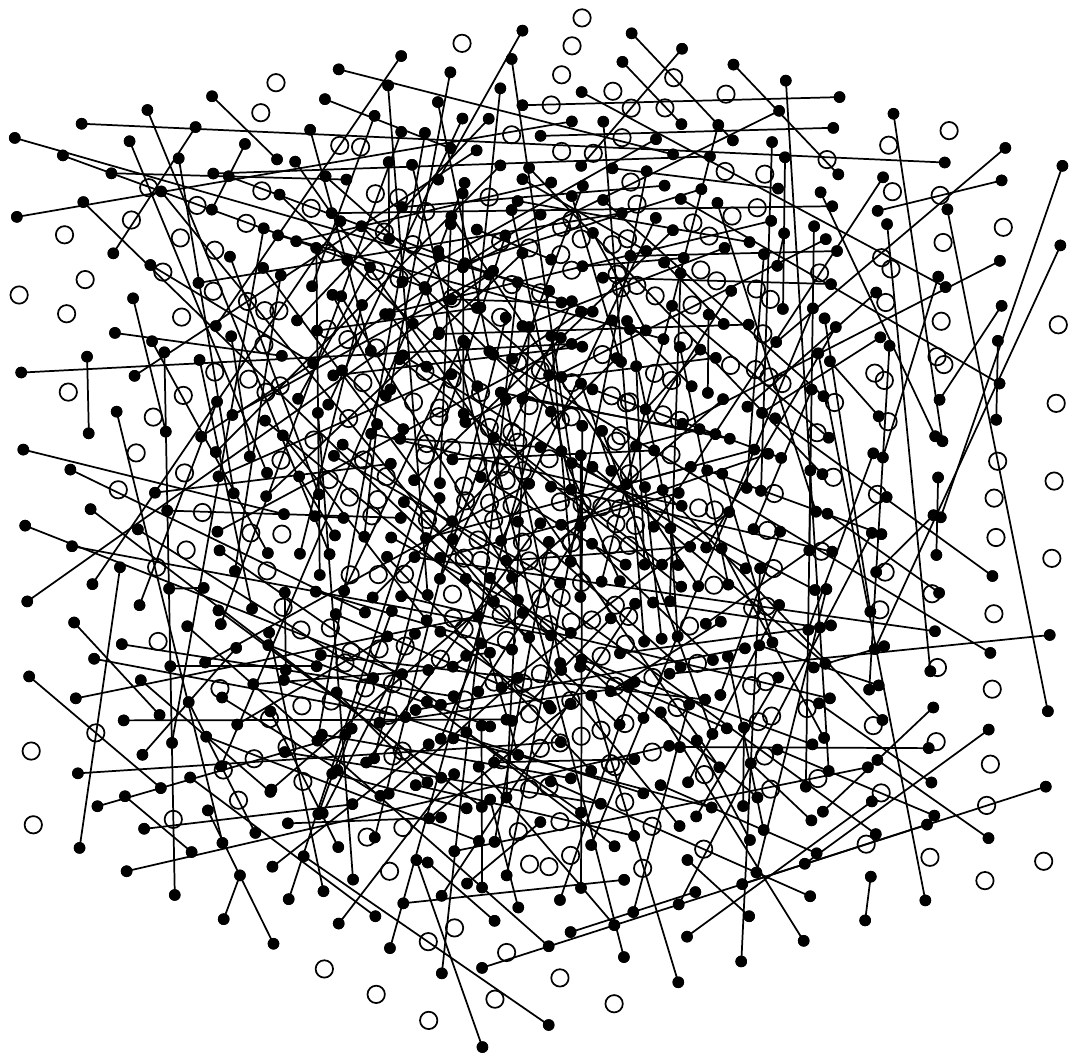}
    \vskip-1cm
  \caption{
    Examples of configurations of the elastic dimer-monomer system in dimensions one, two and three (left to right). Monomers are depicted as ``$\circ$''s. Dimer particles are represented by ``$\bullet$''s, whilst the bonds between them  are shown as arcs in the one-dimensional case and as line segments in the higher dimensions.
    }
    \label{fig:edm_confs}
\end{figure}
In this more general setting, the partition function of the \emph{elastic dimer-monomer} (EDM) system and its thermodynamic limit can be  encoded in the product moments and the MLE of an auxiliary homogeneous GRF whose mean value and the covariance function are the Boltzmann factors associated with the monomer energy and the dimer potential, respectively.
The structure of the covariance function singles out a class of EDM systems where the dimer potential is a weighted $\ell_1$-norm (also known as ``Manhattan norm'') of the dimer vector. A remarkable feature of the  \emph{Manhattan dimer potential} is that the corresponding auxiliary GRF is Markov and, moreover, belongs to the class of Pickard random fields (PRFs) \cite{P_1977,P_1980,TP_1992}; see also, \cite{CIG_1998,CI_2000}. PRFs were developed originally for the planar case and subsequently extended to three dimensions \cite{IG_1999}. Unlike general Markov random fields, they are amenable to a unilateral single-pass simulation, for example,  using the level-sets algorithm of \cite{CD_1998}.  The simplicity of the autoregressive-like modelling allows  the Gaussian PRF (GPRF)  to be utilized for the statistical estimation of the partition function of the Manhattan EDM system, similar to the role of moving average GRFs for the classical monomer-dimer problem. However, in addition to this utility for the alternative Monte-Carlo simulation, we take advantage of the specific Markov structure of the auxiliary GPRF to establish a recurrence relation for its conditional product moments which involves local  transitions from one level set to another. Although the exposition is focused mainly on the two-dimensional case for better visualizability, a recurrence equation for the conditional product moments of the GPRF can also be established in three dimensions. Note that similar recurrence relations are present in the statistical mechanical transfer matrix method for systems with short-range interactions \cite{B_1981,B_1982}. Since the interactions in the Manhattan EDM system are of long-range nature, the possibility of recursive representation of its partition function seems striking.





For the one-dimensional Manhattan EDM system, where the dimer potential is proportional to the standard Euclidean distance, the auxiliary GRF becomes a homogeneous Gaussian Markov sequence satisfying an autoregressive equation. Due to the Markov property, the MLE is reduced to the largest eigenvalue of a compact transfer operator on an invariant cone in a Hilbert space which governs a recurrence equation for the conditional product moments. We obtain a matrix representation of the transfer operator in the basis of Hermite polynomials and provide numerical results on the dependence of the MLE on two parameters of the univariate  Manhattan EDM system. Although this eigenanalysis lies within the general framework of the transfer matrix method, we employ probability theoretic techniques and constructs such as measure change and conditional Gaussian distributions. This allows the transfer operator to be linked with the annihilation and creation operators of the quantum harmonic oscillator \cite[p.~91]{S_1994}, and the eigenvalue problem to be recast into a pantograph functional-differential equation \cite{BDMO_2007,D_1990,KM_1971,M_1940,S_1995} which involves scaling of the independent variable.





Note that the approach to the monomer-dimer problem through the product moments and MLEs of auxiliary GRFs, proposed in the present paper, is different from, and in fact, more general than, the classical methods based on Pfaffians of matrices (see \cite{B_1982} and references therein) and the links of the problem  with the determinants of random matrices \cite{L_2002}. Our machinery differs also from the Gaussian integrals in Grassmann variables and their connections with the Pfaffians \cite{HP_1994}.

The paper is organised as follows. Section~\ref{sec:EDM_system}
describes the class of EDM systems being considered and formulates the problem of computing the equilibrium free energy in the thermodynamic limit. In Section~\ref{sec:connection}, this problem is related to the product moments and the MLE of a homogeneous GRF. Section~\ref{sec:rigid} shows that the auxiliary GRF for the classical rigid dimer problem is of moving average type. Section~\ref{sec:QEDM}  relates the  family of Manhattan EDM systems with GPRFs and focuses on the two-dimensional case. Section~\ref{sec:cond_prod_moms_GPRF} obtains a recurrence relation for their conditional product moments associated with aggregated level sets of Section~\ref{sec:aggregation}. Section~\ref{sec:1d_case} proceeds to one-dimensional Manhattan EDM systems. Section~\ref{sec:cond_prod_moms} derives the transfer operator which governs the recurrence equation for the conditional product moments of the auxiliary Gaussian Markov chain. Section~\ref{sec:upper} establishes an inequality for the norm of this  operator and a related upper bound for the MLE. Section~\ref{sec:matrix} considers the matrix representation of the transfer operator in the basis of Hermite polynomials and discusses its links with the ladder operators of the quantum harmonic oscillator. Section~\ref{sec:pantograph} relates the MLE for the univariate Manhattan EDM system with the eigenvalue  problem for a pantograph functional-differential equation. Section~\ref{sec:conclusion} provides concluding remarks.



\section{Elastic dimer-monomer problem}\label{sec:EDM_system}

Consider a system of identical particles residing at sites of a nonempty finite subset $\Lambda$ of the $n$-dimensional integer lattice $\mZ^n$, with each site of $\Lambda$ being occupied by exactly one particle.
For some disjoint two-element subsets $\{x,y\}\subset \Lambda$, the particles at the lattice sites $x$ and $y$ are (chemically) bonded and form a \emph{dimer}.  The particles which do not participate in dimers are interpreted as \emph{monomers}.

The monomers and dimers are endowed with particular values of energy. The individual energies of monomers are all equal and their common value is denoted by $V$.  The energy $W(z)$ associated with a dimer with endpoints $x\ne y$ is that of the interaction between the particles in it and depends on their relative position $z:=x-y$. Here, $W: \mZ^n\setminus \{0\} \to \mR \bigcup \{+\infty\}$ is a symmetric function which we will refer to as the \emph{dimer potential}.

The classical monomer-dimer formulation \cite{F_1961,FT_1961,K_1961} allows dimers to consist of nearest neighbours only, so that all dimers have a fixed length and the energy of a dimer can only  depend on its orientation.
The  ``elastic'' dimer-monomer (EDM) setting, outlined above, is more general. Indeed,  it still leaves room for geometric constraints which can be imposed by putting $W(z):=+\infty$ for prohibited values of the dimer vector $z$. On the other hand, for admissible $z$, where $W(z)$ is finite, the dimer potential $W$ plays the role of a scoring function which favors dimers with smaller energy, though, in general, it does  not forbid arbitrarily distant particles to form a bond. As in the classical rigid setting, the sites occupied by monomers can also be interpreted as vacancies in which case the EDM system provides a lattice model for a gas of dimers. We will not, however, employ this alternative interpretation in what follows.

A configuration $\omega$ of the EDM system is completely specified by a subset $D$ of evenly many $
    \#D = 2r
$
 sites of $\Lambda$,
and a partition $\{\{x_1, y_1\}, \ldots, \{x_r,y_r\}\}$ of $D$ into two-element subsets which represent  dimers. The class of pair partitions of $D$ is denoted by $\cP_D$, so that $\# \cP_D = (2r-1)!!$. With the complementary set $\Lambda \setminus D$ being occupied by monomers, the total energy of the particle system in the configuration $\omega$ is
\begin{equation}
\label{U}
    U(\omega)
    :=
    (N - 2r) V
    +
    \sum_{k=1}^{r} W(x_k-y_k),
\end{equation}
where
$
    N:= \# \Lambda
$.
The class $\Omega_{\Lambda} := \{\omega\}$ of all possible configurations of the EDM system associated with the set $\Lambda$ consists of
\begin{eqnarray}
\nonumber
    \#\Omega_{\Lambda}
    &=&
    \sum_{r = 0}^{\lfloor N/2\rfloor}
    \begin{pmatrix}
        N\\
        2r
    \end{pmatrix}
    (2r-1)!!
    =
    \bE((1+\zeta)^N)\\
\label{Omegacard}
    &=&
    \d_{\lambda}^N
    \bE \re^{\lambda(1+\zeta)}\big|_{\lambda = 0}
    =
    \d_{\lambda}^N
    \re^{\lambda+\lambda^2/2}\big|_{\lambda = 0}
    =
    i^N H_N(-i)
\end{eqnarray}
partitions of $\Lambda$ into two-element subsets and
singletons. Here,
$\lfloor \cdot \rfloor$ is the floor function, $\zeta$ is a standard normal (that is, zero mean unit variance Gaussian) random variable with the probability density function (PDF)
\begin{equation}
\label{nu}
    \nu(x)
    :=
    \re^{
        -x^2/2}\big/\sqrt{2\pi},
\end{equation}
and
\begin{equation}
\label{H}
    H_k(x)
    :=
    (-1)^k
    \re^{x^2/2}
    \d_x^k
    \re^{-x^2/2}
\end{equation}
is the $k$th modified Hermite polynomial. In (\ref{Omegacard}), use is made of the even order moments $\bE(\zeta^{2r})=(2r-1)!!$ and the moment-generating function $\bE \re^{\lambda \zeta} = \re^{\lambda^2/2}$ for the standard normal distribution in combination with the generating function  of the Hermite polynomials \cite[Theorem 2.5.6 on p.~7]{M_1997}:
\begin{equation}
\label{Hermgen}
    \sum_{k\> 0}
    H_k(x)y^k/k!
    =
    \re^{xy-y^2/2}.
\end{equation}
%
%
%
 By the variational postulate of Statistical
Mechanics describing the canonical ensemble \cite{H_1987}, the particle system,  in contact with a heat bath at absolute temperature $T>0$, acquires an equilibrium probability distribution over  the configurational space $\Omega_{\Lambda}$ that delivers the minimum value
\begin{equation}
\label{F_Lambda}
    F_{\Lambda}
    :=
    \min_P
    \sum_{\omega\in \Omega_{\Lambda}}
    (
        U(\omega)
        +
        T
        \ln P(\omega)
    )
    P(\omega)
    =
    -T\ln Z_{\Lambda}
\end{equation}
to the free energy functional
over all probability mass functions (PMFs) $P:
\Omega_{\Lambda}\to [0,1]$ satisfying $\sum_{\omega \in \Omega_{\Lambda}}
P(\omega) = 1$. By the finiteness of the set $\Omega_{\Lambda}$ and the  strict convexity of the functional being minimized in (\ref{F_Lambda}) over the simplex  of PMFs,  the equilibrium value $F_{\Lambda}$  of the free energy is
achieved at the unique Gibbs-Boltzmann PMF
$
    P_{\Lambda}(\omega)
    =
    \re^{-\beta U(\omega)}/ Z_{\Lambda}$.
Here,
\begin{equation}
\label{beta}
    \beta
    :=
    1/T,
\end{equation}
where the units of the temperature $T$ are chosen so
as to ``absorb'' the Boltzmann constant, and
\begin{equation}
\label{Z}
    Z_{\Lambda}
    :=
    \sum_{\omega\in \Omega_{\Lambda}}
    \re^{
        -\beta U(\omega)
    }
\end{equation}
is the canonical partition function \cite{H_1987,ME_1981} which encodes the
equilibrium properties  of the system. For example, the average total energy of the particle system is recovered from (\ref{Z}) as
$
    \sum_{\omega \in \Omega_{\Lambda}}
    U(\omega)
    P_{\Lambda}(\omega)
    =
    -\d_{\beta} \ln Z_{\Lambda}
$.
We will be concerned with the statistical mechanical problem of computing the partition function $Z_{\Lambda}$ and the thermodynamic limit of
the equilibrium free energy (\ref{F_Lambda}) per site of an
unboundedly increasing fragment $\Lambda$ of the lattice $\mZ^n$:
\begin{equation}
\label{F}
    \lim_{\Lambda\to \infty}
    \frac{F_{\Lambda}}{\#\Lambda}
    =
    -T
    \lim_{\Lambda \to \infty}
    \frac{\ln Z_{\Lambda}}{\# \Lambda}.
\end{equation}
Although the convergence $\Lambda\to \infty$ is usually understood in the sense of van Hove
\cite{R_1978}, we   will often assume, for simplicity, that
$\Lambda$ is a discrete hypercube $\{0, \ldots, N-1\}^n$ (or another simply shaped subset of the lattice), where $N$ is an integer parameter which tends to $+\infty$.

\section{Auxiliary Gaussian random field}\label{sec:connection}

Suppose the temperature $T$ is fixed and the
dimer potential $W$ is extended to the origin so as to satisfy the condition $W(0)\ne \infty$ and the inequality
\begin{equation}
\label{dom}
    \re^{-\beta W(0)}
    \>
    \sum_{z\in \mZ^n\setminus \{0\}}
    \re^{-\beta W(z)}.
\end{equation}
To make this possible, we assume that the sum on the right-hand side of (\ref{dom}) is finite. Then the function
$C: \mZ^n \to \mR_+$, defined as the Boltzmann factor associated with the dimer potential $W$ by
\begin{equation}
\label{C}
    C(z)
    :=
    \re^{-\beta W(z)},
\end{equation}
with the shorthand notation (\ref{beta}),
is positive definite \cite{GS_2004}. That is, for any  positive integer $m$ and any points $z_1, \ldots, z_m$ in $\mZ^n$, the matrix $(C(z_j-z_k))_{1\< j,k\< m}$ is positive semi-definite. The fact that (\ref{dom}) implies the positive definiteness of $C$ in (\ref{C}) is established by a diagonal dominance argument \cite[p.~349]{HJ_2007}.
Hence, by Bochner's theorem,
$C$ is the covariance function of a real-valued  homogeneous random field on $\mZ^n$ which can be chosen to be Gaussian (without additional assumptions on $C$). Therefore,
 the fulfillment of (\ref{dom}) ensures the existence  of a homogeneous Gaussian random field (GRF) $\xi :=
(\xi_x)_{x \in \mZ^n}$ satisfying
\begin{equation}
\label{Ecov}
    \bE \xi_x
    =
    \re^{-\beta V},
    \qquad
    \cov(\xi_x, \xi_y)
    =
    C(x-y)
\end{equation}
for all $x, y \in \mZ^n$. Here,  $\cov(\cdot, \cdot)$ is the covariance of random variables, $V$ is the monomer energy, and $C$ is given by (\ref{C}). In what follows, such a random field
$\xi$ will be referred to as
the \emph{auxiliary GRF}.  Its significance for
the EDM system, described in Section~\ref{sec:EDM_system}, is
clarified by the theorem below.

\begin{theorem}
\label{th:connection}
Suppose the dimer potential $W$ satisfies (\ref{dom}). Then for any finite set $\Lambda\subset \mZ^n$, the partition function (\ref{Z})
coincides with the product moment (\ref{bM}) of the auxiliary GRF
$\xi$ over $\Lambda$:
\begin{equation}
\label{Zxi}
    Z_{\Lambda}
    =
    \bM_{\Lambda}(\xi).
\end{equation}
\end{theorem}
\begin{proof}
The product moment of the auxiliary GRF $\xi$ over the set $\Lambda$ can be  computed as
\begin{equation}
\label{MLambdaxi}
    \bM_{\Lambda}(\xi)
    =
    \sum_{D\subset \Lambda}
    \mu^{N - \# D}
    \bM_D(\eta)
    =
    \sum_{r=0}^{\lfloor N/2\rfloor}
    \mu^{N-2r}
    \sum_{D\subset \Lambda:\, \# D = 2r}
    \bM_D(\eta),
\end{equation}
where $N := \# \Lambda$, and $\eta:= (\eta_x)_{x\in \mZ^n}$ is a GRF obtained by centering $\xi$ as $\eta_x := \xi_x - \mu$, with
\begin{equation}
\label{mu}
    \mu
    :=
    \re^{-\beta V},
\end{equation}
in conformance with (\ref{Ecov}).
By the Isserlis-Wick theorem \cite{I_1918}, which relates mixed moments of evenly
many zero mean Gaussian random variables with their cross-covariances,\footnote{and is used as an algebraic moment closure in the statistical theory of turbulence \cite[pp. 44--45]{F_1995} and quantum field theory (see also \cite[Theorem 1.28 on pp. 11--12]{J_1997})} it follows that
\begin{equation}
\label{MDeta}
    \bM_D(\eta)
    =
    \sum_{\cP_D}
    \prod_{k=1}^{r}
    \cov(\xi_{x_k}, \xi_{y_k})
    =
    \sum_{\cP_D}
    \prod_{k=1}^{r}
    C(x_k- y_k).
\end{equation}
Here, the sum extends over the class $\cP_D$ of partitions $\{\{x_1, y_1\}, \ldots, \{x_r, y_r\}\}$ of the set $D$ into two-element subsets, with $r:= \#D/2$. Now, upon substitution of (\ref{C}) into (\ref{MDeta}), (\ref{MLambdaxi}) takes the form
\begin{align}
\nonumber
    \bM_{\Lambda}(\xi)
    =&
    \sum_{r=0}^{\lfloor N/2\rfloor}
    \re^{-(N-2r)\beta V}
    \sum_{D\subset \Lambda:\, \#D= 2r}
    \exp
    \left(
        -\beta\sum_{k=1}^{r} W(x_k-y_k)
    \right)\\
\label{MLambdaxi1}
    =&
    \sum_{D\subset \Lambda:\, \#D\, {\rm is\, even}}\
    \sum_{\cP_D}
    \exp
    \left(
        -\beta
        \left(
            (N-\#D)V
            +
            \sum_{k=1}^{\#D/2}
            W(x_k-y_k)
        \right)
    \right).
\end{align}
In view of (\ref{U}) and the definition of the configurational space $\Omega_{\Lambda}$ of the EDM system from Section~\ref{sec:EDM_system},  the right-hand side of (\ref{MLambdaxi1}) coincides with the partition function $Z_{\Lambda}$ in (\ref{Z}), thus proving (\ref{Zxi}).
\end{proof}

In view of Theorem~\ref{th:connection}, the rightmost limit in  (\ref{F}) coincides with the moment Lyapunov
exponent (MLE) of the auxiliary GRF $\xi$:
\begin{equation}
\label{MLE}
    \lim_{\Lambda \to \infty}
    \frac{\ln Z_{\Lambda}}{\# \Lambda}
    =
    \lim_{\Lambda \to \infty}
    \frac
    {\ln \bM_{\Lambda}(\xi)}
    {\# \Lambda}
    =:
    \bm(\xi),
\end{equation}
which is identical to
(\ref{bm}) since the product moments of the GRF being considered are all nonnegative.
This limit is completely specified by the mean $\mu$, given by (\ref{mu}), and the spectral density function $S: [-\pi,\pi]^n\to \mR_+$ defined as the Fourier transform of the covariance function (\ref{C}) by
\begin{equation}
\label{S}
    S(\lambda)
    :=
    \sum_{z\in \mZ^n}
    C(z)\re^{-i\lambda^{\rT} z}
    =
    \sum_{z\in \mZ^n}
    \re^{-\beta W(z)-i\lambda^{\rT} z}.
\end{equation}
Here, $u^{\rT}v$ is the inner product in $\mR^n$,  vectors are organised as columns unless indicated otherwise, and $(\cdot)^{\rT}$ denotes the transpose.
Computing the MLE $\bm(\xi)$, with $\bM_{\Lambda}(\xi)$ in  (\ref{MLE}) being replaced with $|\bM_{\Lambda}(\xi)|$,  is of considerable interest in its own right for a general homogeneous GRF $\xi$.  However, the  salient feature of this problem in the context of the EDM system is that, in view of (\ref{C}), both the mean value and the covariance function of the auxiliary GRF in (\ref{Ecov}) are nonnegative. The nonnegativeness of $C$ implies that the spectral density $S$ in (\ref{S}) is not only symmetric and nonnegative,   but is also positive definite. Hence, $S$ is the  covariance function of yet another homogeneous GRF, this time, on the $n$-dimensional torus $[-\pi,\pi)^n$, whose probability distributions are invariant under the group of entrywise additions modulo $2\pi$.
\begin{remark}
The fact that the choice of $W(0)$,  which affects $C(0)$ in  (\ref{C}), has no influence on the product moments of the auxiliary GRF $\xi$ is explained by their invariance under adding a $\xi$-independent ``white noise'' GRF $\zeta := (\zeta_x)_{x\in \mZ^n}$ whose values $\zeta_x$ are independent Gaussian random variables with zero mean and variance $\gamma^2$. The resulting GRF $\xi+\zeta:= (\xi_x+\zeta_x)_{x\in \mZ^n}$ has spectral density $S + \gamma^2$ and its product moment over a finite set $\Lambda \subset \mZ^n$ is computed as
\begin{equation}
\label{white}
    \bM_{\Lambda}(\xi+\zeta)
    =
    \bE
    \bM_{\Lambda}(\xi+\zeta \mid \xi_{\Lambda})
    =
    \bE
    \prod_{x\in \Lambda}
    (\xi_x + \bE \zeta_x)
    =
    \bM_{\Lambda}(\xi).
\end{equation}
Here, $\xi_{\Lambda}:= (\xi_x)_{x\in \Lambda}$ is the restriction of $\xi$ to the set $\Lambda$, and, in extending (\ref{bM}), use is made of the \emph{conditional product moment}
$
    \bM_{\Lambda}(\eta\mid \fS)
    :=
    \bE
    \left(
        \prod_{x\in \Lambda}
        \eta_x
        \mid
        \fS
    \right)
$
of a random field $\eta:= (\eta_x)_{x\in \mZ^n}$ over the set $\Lambda$ with respect to a $\sigma$-subalgebra $\fS$ (or a conditioning random element which  generates $\fS$).
 In (\ref{white}), we have also used the tower property of iterated conditional expectations \cite{Sh_1995}, and the assumption that $\zeta$ consists of independent zero mean random variables, independent of $\xi$.
Therefore, under the transformation  $\mu\mapsto \rho \mu$ and  $S\mapsto \rho^2S + \gamma^2$, which corresponds to $\xi\mapsto \rho \xi + \zeta$ with a constant factor $\rho >0$, the MLE $\bm(\xi)$ in (\ref{MLE}), as a functional of $\mu$ and the  spectral density  $S$ from (\ref{S}),   is transformed as  $\bm \mapsto \bm + \ln \rho$.
\end{remark}


\section{The classical monomer-dimer problem and moving average Gaussian random fields}\label{sec:rigid}

Consider the structure of the auxiliary GRF for the classical monomer-dimer problem \cite{F_1961,FT_1961,K_1961} as a particular case of EDM systems.
Suppose the dimer potential $W$ on the set $\mZ^n\setminus \{0\}$ is given by
\begin{equation}
\label{Wdimer}
    W(z)
    :=
    \left\{
        \begin{array}{ccl}
        \alpha_k & {\rm if} & z = \pm e_k\\
        +\infty & {\rm if} & |z|>1
        \end{array}
    \right.,
\end{equation}
where $\alpha_k$ is the energy ascribed to dimers which are parallel to the $k$th coordinate axis, and
\begin{equation}
\label{ek}
    e_k
    :=
    (\underbrace{0, \ldots,0}_{k-1\ {\rm zeros}}, 1, \underbrace{0,\ldots,0}_{n-k\ {\rm zeros}})^{\rT}
\end{equation}
is the $k$th basis vector in $\mR^n$. 
Note that (\ref{Wdimer}) forbids dimers $\{x,y\}$ with $|x-y|>1$ and describes the dependence of the energy of a dimer of Euclidean length 1 on its orientation. In particular, if $V:= +\infty$ (so that monomers are energetically forbidden) and $\alpha_1 = \ldots = \alpha_n=0$, then the partition function $Z_{\Lambda}$ in
(\ref{Z}) in this monomer-free isotropic case counts the number of ``domino'' tilings  \cite[p.~125]{B_1982} of the set
$\Lambda$, that is, partitions of $\Lambda$ into pairs of nearest
neighbours, assuming that $\#\Lambda$ is even. In general, the dimer orientations with smaller values of $\alpha_k$ are energetically favoured. The resulting series on the right-hand side of (\ref{dom}) consists of finitely many nonzero terms, and the auxiliary GRF $\xi$ for this classical ``rigid'' dimer
system can be explicitly constructed as follows. Using (\ref{ek}), for every $k = 1, \ldots, n$, we  denote by
\begin{equation}
\label{Ek}
    E_k
    :=
    \mZ^n + e_k/2
\end{equation}
the set of midpoints of the edges between neighbouring sites  in $\mZ^n$ parallel to the $k$th coordinate axis. For any point $x \in \mZ^n$, the nearest sites of $E_k$ are $x\pm e_k/2$, both at the Euclidean distance $1/2$ from $x$. Accordingly, those midpoints from the combined set
\begin{equation}
\label{E}
    E
    :=
    \bigcup_{k=1}^n E_k,
\end{equation}
which are nearest to a given $x\in \mZ^n$, are $x\pm e_k/2$, with $k=1, \ldots, n$. We will now
associate
independent standard normal random variables $\eta_y$ with the sites
$y \in E$, and ascribe,  to every $x:= (x_k)_{1\< k \< n} \in \mZ^n$, a weighted  sum
\begin{equation}
\label{MA}
    \xi_x
    :=
    \mu
    +
    \sum_{k=1}^{n}
    \sqrt{\rho_k}
    (\eta_{x-e_k/2} + \eta_{x+e_k/2}),
\end{equation}
over the $2n$ sites of the set (\ref{E}), nearest to $x$,
where $\rho_1, \ldots, \rho_n$ are nonnegative parameters given by
\begin{equation}
\label{rhok}
    \rho_k
    :=
    \re^{-\alpha_k \beta}.
\end{equation}
 The resulting GRF $\xi:=(\xi_x)_{x\in
\mZ^n}$ is homogeneous and has mean $\mu$ and covariance function
\begin{equation}
\label{MAC}
    C(z)
    :=
    \cov(\xi_0,\xi_z)
    =
    \left\{
    \begin{array}{cl}
    2\sum_{k=1}^{n}
    \rho_k & {\rm if}\ z=0\\
    \rho_k & {\rm if}\ z = \pm e_k\\
    0& {\rm otherwise}
    \end{array}
    \right.,
\end{equation}
which, in view of (\ref{rhok}), is related with the dimer potential (\ref{Wdimer}) by (\ref{C}), with $W$ being extended to the origin as
$
    W(0)
    =
    -T\ln
    (
        2\sum_{k=1}^n
        \rho_k
    )
$.
Therefore, (\ref{MA}) indeed produces an auxiliary  GRF $\xi$ for the classical dimer system. It is a moving average random field which is easy to simulate. The simulation boils down to generating the standard normal random variables at sites of the set $E$ given by (\ref{Ek})--(\ref{E}). This, in principle, provides an alternative probabilistic way for approximate computation of the partition function $Z_{\Lambda}$ for the monomer-dimer problem over a finite set $\Lambda\subset \mZ^n$  by modelling the auxiliary GRF $\xi$ and statistically estimating its product moment $\bM_{\Lambda}(\xi)$ instead of the Monte-Carlo simulation of the underlying particle system which is more complicated.

In view of (\ref{MAC}), the covariance matrix $\cov(\xi_{\Lambda})$ of the restriction of $\xi$ to a finite set $\Lambda\subset \mZ^n$, with $\xi_{\Lambda}$ being considered as a Gaussian random vector of dimension $\# \Lambda$, has a $(2n+1)$-diagonal structure. A similar sparsity structure is known for the inverse covariance matrices (that is, precision matrices) of finite fragments of Markov GRFs \cite{RH_2006}. The Markov GRFs   describe the equilibrium states of systems of linear harmonic oscillators with nearest neighbour interaction and are, in general, hard to simulate in contrast to the moving average random fields above; see also \cite{B_1974,G_1989}.  An important subclass of Markov GRFs which is amenable to efficient unilateral simulation in a single pass through the simulation domain is provided by Pickard random fields (PRFs) \cite{CIG_1998,CI_2000,IG_1999,P_1977,P_1980,TP_1992} which behave in space like autoregressive processes do in time.

\section{Manhattan dimers and Gaussian Pickard random fields}\label{sec:QEDM}

Suppose the dimer potential is given by a weighted $\ell_1$-norm (also called ``Manhattan norm'') of the dimer vector $    z
    :=
    (z_k)_{1\< k \< n} \in \mZ^n
$:
\begin{equation}
\label{manpot}
    W(z)
    :=
    \sum_{k=1}^{n}
    \alpha_k |z_k|.
\end{equation}
Here, $\alpha_1, \ldots, \alpha_n$ are positive parameters which quantify the projections of the force of attraction between the particles in a dimer onto the coordinate axes. Unlike ideal springs, the hypothetical \emph{Manhattan dimer} does not develop a linearly increasing force when being stretched, albeit the attraction between the particles in it persists. The covariance function (\ref{C}), associated with  the Manhattan dimer potential (\ref{manpot}), is
\begin{equation}
\label{mancov}
    C(z)
    =
    \prod_{k=1}^n
    \rho_k^{|z_k|},
\end{equation}
where $0<\rho_1, \ldots, \rho_n<1$ are computed according to (\ref{rhok}). The resulting auxiliary GRF $\xi$ is a Gaussian PRF (GPRF) in dimensions two \cite{TP_1992} and three \cite{IG_1999} and lends itself to the unilateral simulation mentioned above. Such simulation of the GPRF $\xi$ is made possible by the following enhancement of the spatial Markov property for any $x:= (x_k)_{1\< k \< n}\in \mZ^n$:
\begin{equation}
\label{GPRF}
    \[[[\xi_x\mid \xi_{\Lambda_x^-}\]]]
    =
    \[[[\xi_x\mid \xi_{\Lambda_x^0}\]]].
\end{equation}
Here, $\[[[\eta \mid \zeta\]]]$ is a shorthand notation\footnote{We use $\[[[\eta \mid \zeta\]]]$ instead of the more standard notation $\sP_{\eta\mid \zeta}$ in order to mitigate the burden of multitiered  subscripts.} for the conditional probability distribution of $\eta$ given $\zeta$, and
\begin{equation}
\label{order}
    \Lambda_x^-
    :=
    \mZ^n
    \setminus
    \Lambda_x^+,
    \quad
    \Lambda_x^+
    :=
    \mZ^n
    \bigcap
    \mathop{\x}_{k=1}^{n}
    [x_k, +\infty),
    \quad
    \Lambda_x^0
    :=
    \left(\mathop{\x}_{k=1}^{n}
    \{x_k-1,x_k\}\right)
    \setminus \{x\}.
\end{equation}
The sets $\Lambda_x^-$ and $\Lambda_x^+$ are the ``past'' and ``future'' of the lattice with respect to the site $x$ in the spatial sense \cite{CI_2000}. Accordingly, the set $\Lambda_x^0$, which consists of $2^n-1$ points from $\Lambda_x^-$,   plays the role of the nearest past for $x$; see Fig.~\ref{fig:order}.
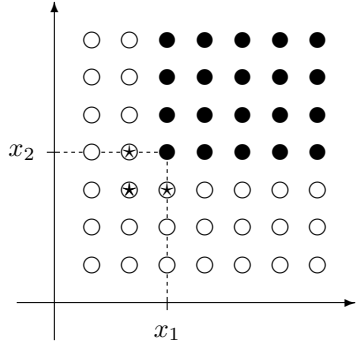
\begin{figure}[htbp]
\centering
\unitlength=1.0mm
\linethickness{0.4pt}
\begin{picture}(50.00,45.00)
    \matrixput(20,25)(5,0){5}(0,5){4}{\circle*{2}}

    \matrixput(10,25)(5,0){2}(0,5){4}{\circle{2}}

    \matrixput(20,10)(5,0){5}(0,5){3}{\circle{2}}

    \matrixput(10,10)(5,0){2}(0,5){3}{\circle{2}}

    \multiput(15,20)(5,0){2}{\put(0,0){\makebox(0,0)[cc]{$\star$}}}
    \multiput(15,20)(0,5){2}{\put(0,0){\makebox(0,0)[cc]{$\star$}}}

    \put(0,5){\vector(1,0){45}}
    \put(5,0){\vector(0,1){45}}

    \put(20,4){\line(0,1){1.75}}
    \makebox(40,2)[ct]{$x_1$}

    \multiput(-20,5)(0,5){4}{\multiput(0,1.25)(0,1){3}{\line(0,1){0.5}}}

    \put(-36,25){\line(1,0){1.75}}
    \makebox(-37.5,50)[rc]{$x_2$}
    \multiput(-35,25)(5,0){3}{\multiput(1.25,0)(1,0){3}{\line(1,0){0.5}}}
\end{picture}
\caption{
    A fragment of the two-dimensional integer lattice $\mZ^2$. The sites of the spatial past $\Lambda_x^-$ and  the spatial future $\Lambda_x^+$ for a given site $x\in \mZ^2$, defined by (\ref{order}), are represented by ``$\circ$''s and ``$\bullet$''s, respectively. The sites of the nearest past $\Lambda_x^0$ are marked by ``$\star$''s.
}
\label{fig:order}
\end{figure}
In the one-dimensional case $n=1$, which we will study in Sections \ref{sec:1d_case}--\ref{sec:pantograph}, the relation (\ref{GPRF}) between the conditional probability laws reduces to the standard Markov property, thus allowing   $\xi$ (which becomes a Gaussian random sequence) to be generated by an autoregressive equation.

A remarkable feature of the Markov structure of the GPRF in the multivariate case $n> 1$ is that its restrictions $\xi_{L_{\ell}}$ to the \emph{level sets} $L_0, L_1, L_2, \ldots $, described below, form a Markov chain of order $n$. This property is employed by the level-sets algorithm of \cite{CD_1998} for the unilateral  simulation of such random fields on the nonnegative orthant  $\mZ_+^n$ of the lattice, where $\mZ_+$ is the set of nonnegative   integers.  For example, in the two-dimensional case $n=2$ elucidated by Fig.~\ref{fig:order},  the  level sets  are given  by
\begin{equation}
\label{L}
    L_{\ell}
    :=
    \{
                (j,k) \in \mZ_+^2:\
                j+k = \ell
    \},\
    \qquad
    \ell \> 0,
\end{equation}
and the restrictions $\xi_{L_{\ell}}$ of the GPRF, which are organised as Gaussian random vectors (with $\xi_{L_{\ell}}$ having dimension $\ell+1$), form a Markov chain of order two. More precisely,  for any $\ell >0$, the conditional probability distribution $\[[[ \xi_{L_{\ell+1}} \mid \xi_{L_0}, \ldots, \xi_{L_{\ell}}\]]]$ coincides with $\[[[ \xi_{L_{\ell+1}} \mid \xi_{L_{\ell-1}}, \xi_{L_{\ell}}\]]]$. Furthermore, the values of $\xi$ at sites of the next level set $L_{\ell+1}$ are conditionally independent given $\xi_{L_{\ell-1}}$ and $\xi_{L_{\ell}}$, with their conditional joint probability distribution
\begin{equation}
\label{prod}
    \[[[ \xi_{L_{\ell+1}} \mid \xi_{L_{\ell-1}}, \xi_{L_{\ell}}\]]]
    =
    \mathop{\x}_{(j,k)\in L_{\ell+1}}
    \[[[
        \xi_{jk} \mid             \xi_{\wh{\Lambda}_{jk}^0}
    \]]]
\end{equation}
being the product of $\ell+2$ conditional Gaussian distributions. Here, $\xi_{jk}$ at site $(j,k)$ is conditioned on the values of $\xi$ in  the nearest past  $\Lambda_{jk}^0$ which is restricted to the nonnegative quadrant $\mZ_+^2$ of the lattice as
\begin{equation}
\label{hatLambda}
    \wh{\Lambda}_{jk}^0
    :=
    \Lambda_{jk}^0
    \bigcap
    \mZ_+^2
    =
    \left\{
        \begin{array}{cl}
            \{(j-1,k-1),\ (j,k-1),\ (j-1,k)\} & {\rm for}\ j>0,\ k>0\\
            \{(j-1,0)\} & {\rm for}\ j>0,\ k=0            \\
            \{(0,k-1)\} & {\rm for}\ j=0,\ k>0
        \end{array}
    \right.
\end{equation}
in conformance with (\ref{order}) for the two-dimensional case. In terms of the centered random field $\sigma:= (\sigma_x)_{x \in \mZ^2}$, defined by
\begin{equation}
\label{sigma}
    \sigma_x
    :=
    \xi_x - \mu,
\end{equation}
with $\mu$ associated with the monomer energy by (\ref{mu}) as before, the conditional Gaussian distributions on the right-hand side of (\ref{prod}) are given by
\begin{equation}
\label{cond_gauss}
    \[[[
        \sigma_{jk}
        \mid
        \sigma_{\wh{\Lambda}_{jk}^0}
    \]]]
     =
    \left\{
        \begin{array}{cl}
            \cN(
                a \sigma_{j-1,k-1}
                +
                b \sigma_{j,k-1}
                +
                c\sigma_{j-1,k},\,
                d^2
            ) & {\rm for}\ j>0,\ k>0\\
            \cN(
                \rho_1 \sigma_{j-1,0},\,
                \gamma_1^2
            ) & {\rm for}\ j>0,\ k=0            \\
            \cN(
                \rho_2 \sigma_{0,k-1},\,
                \gamma_2^2
            ) & {\rm for}\ j=0,\ k>0
        \end{array}
    \right.,
\end{equation}
where $\rho_1$, $\rho_2$ are related to the attraction force parameters $\alpha_1$, $\alpha_2$ of the Manhattan dimer potential (\ref{manpot}) by (\ref{rhok}), and the conditional variances $\gamma_1^2$, $\gamma_2^2$ of the last two distributions are computed as
\begin{equation}
\label{gamma12}
    \gamma_1
    :=
    \sqrt{1-\rho_1^2},\
    \qquad
    \gamma_2
    :=
    \sqrt{1-\rho_2^2}.
\end{equation}
 This follows from the well-known results on conditional distributions for jointly Gaussian random variables \cite{Sh_1995}.
The coefficients $a$, $b$, $c$ and the conditional variance $d^2$ of the Gaussian distribution for $j,k>0$ in (\ref{cond_gauss}) are computed as
\begin{eqnarray}
\nonumber
    \begin{bmatrix}
        a,\
        b, \
        c
    \end{bmatrix}
    & = &
    \cov
    (
        \sigma_{jk},
        \varpi_{jk}
    )
    \cov
    (
        \varpi_{jk}
    )^{-1}\\
\label{abc}
    & = &
    \begin{bmatrix}
        \rho_1\rho_2, \
        \rho_2, \
        \rho_1
    \end{bmatrix}
    \begin{bmatrix}
        1       & \rho_1    & \rho_2        \\
        \rho_1  & 1         & \rho_1\rho_2  \\
        \rho_2  & \rho_1\rho_2 & 1
    \end{bmatrix}^{-1}
     =
    \begin{bmatrix}
        -\rho_1\rho_2, \
        \rho_2, \
        \rho_1
    \end{bmatrix},\\
    \nonumber\\
\nonumber
    d^2
    & = &
    \var(\sigma_{jk})
    -
    \begin{bmatrix}
        a, \
        b, \
        c
    \end{bmatrix}
    \cov
    (
        \sigma_{jk},
        \varpi_{jk}
    )^{\rT}    \\
\label{d}
    & = &
    1-
    \begin{bmatrix}
        -\rho_1\rho_2, \
        \rho_2, \
        \rho_1
    \end{bmatrix}
    \begin{bmatrix}
        \rho_1\rho_2\\
        \rho_2\\
        \rho_1
    \end{bmatrix}
    =
    (1-\rho_1^2) (1-\rho_2^2)
    =
    \gamma_1^2 \gamma_2^2,
\end{eqnarray}
with $\var(\cdot)$ the variance of a random variable.
Here, the conditioning restriction $\sigma_{\wh{\Lambda}_{jk}^0}$ of the centered random field (\ref{sigma}) to the set (\ref{hatLambda}) is organized as the three-dimensional  row-vector
$
    \varpi_{jk}
    :=
        \begin{bmatrix}
            \sigma_{j-1,k-1},\
            \sigma_{j,k-1}, \
            \sigma_{j-1,k}
        \end{bmatrix}
$
whose covariance $(3\x 3)$-matrix $\cov(\varpi_{jk})$ and the row-vector $\cov(\sigma_{jk}, \varpi_{jk})$ of cross-covariances with $\sigma_{jk}$ are appropriate submatrices of the covariance matrix
$$
    \cov
    \left(
        \begin{bmatrix}
            \xi_{j-1,k-1}\\
            \xi_{j,k-1}\\
            \xi_{j-1,k}\\
            \xi_{jk}\\
        \end{bmatrix}
    \right)
    =
    \begin{bmatrix}
        1       & \rho_1    & \rho_2        & \rho_1\rho_2\\
        \rho_1  & 1         & \rho_1\rho_2  & \rho_2\\
        \rho_2  & \rho_1\rho_2 & 1 & \rho_1\\
        \rho_1\rho_2  & \rho_2         & \rho_1  & 1
    \end{bmatrix}
$$
obtained by evaluating (\ref{mancov}) for  the square cell $\{j-1, j\}\x \{k-1, k\}$ of the lattice; see also \cite{P_1980,TP_1992}.

\section{Aggregation of level sets}\label{sec:aggregation}

We will now employ the technique of reducing the order of a Markov chain by augmenting its state space.
For any $\ell>0$, consider the restriction of the centered GPRF $\sigma$, defined by (\ref{sigma}), to the union
\begin{equation}
\label{Theta}
    \Theta_{\ell}
    :=
            L_{\ell-1}\bigcup L_{\ell}
\end{equation}
of two adjacent level sets from (\ref{L}) which consists of $2\ell + 1$ lattice sites.  Also, let $\Theta_0:= L_0$ be the singleton formed by the origin of $\mZ^2$. The resulting sequence of random vectors $\sigma_{\Theta_{\ell}}$, labeled by $\ell \> 0$, is a Markov chain of order one.  In view of the reflection symmetry of the covariance function (\ref{mancov}) with respect to the coordinate axes, the restrictions $\sigma_{-\Theta_{\ell}}$ of $\sigma$ to the sets $-\Theta_{\ell}$, which are centrally symmetric to the aggregated level sets (\ref{Theta}) about the origin,  also form a Markov chain. Furthermore, the Markov property remains valid upon removal of the ``endpoints'' $(0,-\ell)$ and $(-\ell,0)$ from the set  $-\Theta_{\ell}$. More precisely, the restrictions $\sigma_{\Gamma_{\ell}}$ of the GPRF $\sigma$ to the sets $\Gamma_0, \Gamma_1, \Gamma_2, \ldots$ defined by
\begin{eqnarray}
\nonumber
    \Gamma_N
    & := &
    (-\Theta_{N+1} )
    \setminus
    \{(0,-N-1),\ (-N-1,0)\}\\
\nonumber
    & = &
    \{(0,-N),\
    (0,-N-1),\
    (-1,-N-1),\ \ldots, \\
\label{Gamma}
    & &
    (-N+1,-1),\
    (-N,-1),\
    (-N, 0)\},
\end{eqnarray}
also form a Markov chain. We label the values of $\sigma$ at sites of the set $\Gamma_N$ by integers $0$ to $2N$ as shown in Fig.~\ref{fig:zigzag},
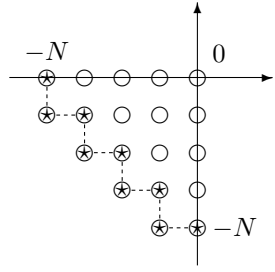
\begin{figure}[htbp]
\centering
\unitlength=1.0mm
\linethickness{0.4pt}
\begin{picture}(50,35)

    \put(10,27){\makebox(0,0)[cb]{$-N$}}

    \put(32,5){\makebox(0,0)[lc]{$-N$}}

    \put(32,27){\makebox(0,0)[lb]{$0$}}

    \multiput(10,25)(5,-5){4}
    {
        \put(0,0){{\circle{2}}}
        \put(0,0){\makebox(0,0)[cc]{$\star$}}
        \multiput(0,-1.25)(0,-1){3}{\line(0,-1){0.5}}

        \multiput(1.25,-5)(1,0){3}{\line(1,0){0.5}}

        \put(0,-5){{\circle{2}}}
        \put(0,-5){\makebox(0,0)[cc]{$\star$}}
    }

    \multiput(15,25)(5,-5){4}{\circle{2}}

    \multiput(20,25)(5,-5){3}{\circle{2}}

    \multiput(25,25)(5,-5){2}{\circle{2}}

    \put(30,25){\circle{2}}

        \put(30,5){{\circle{2}}}
        \put(30,5){\makebox(0,0)[cc]{$\star$}}

    \put(5,25){\vector(1,0){35}}
    \put(30,0){\vector(0,1){35}}

\end{picture}
\caption{
  The subset $\Delta_N$ of the nonpositive quadrant $\mZ_-^2$ of the planar integer lattice, defined by (\ref{Delta}), is represented by ``$\circ$''s for $N=4$. The sites of the set $\Gamma_N$ from (\ref{Gamma}) are marked by ``$\star$''s and labeled by integers $0$ to $2N$ in a ``zigzag'' manner (along the dashed line) starting from the point $(0,-N)$.
}
\label{fig:zigzag}
\end{figure}
which allows $\sigma_{\Gamma_N}$ to be  identified with a $(2N+1)$-dimensional Gaussian random vector
\begin{eqnarray}
\nonumber
    \Sigma_N
     :=
    (\Sigma_{N,k})_{0\< k\< 2N}
    & := &
    \{
        \sigma_{0,-N},\
        \sigma_{0,-N-1},\
        \sigma_{-1,-N-1},\ \ldots, \\
\label{Sigma}
    & &
        \sigma_{-N+1,-1},\
        \sigma_{-N,-1},\
        \sigma_{-N, 0}\}.
\end{eqnarray}
Due to this numbering, $\Sigma_N$ is obtained by sampling the GPRF $\sigma$ along a monotone path in the planar lattice. Hence, by the results of \cite{P_1980,TP_1992}, for every given $N\> 0$, the entries $\Sigma_{N,0}, \ldots, \Sigma_{N,2N}$ of the random vector $\Sigma_N$ form a nonhomogeneous Markov chain. Moreover, by evaluating the covariance function (\ref{mancov}) for the zigzag shaped set $\Gamma_N$, it follows that $\Sigma_N$ is an alternating autoregressive sequence \cite{MSW_2007} governed by the recurrence equation
\begin{equation}
\label{AAR}
    \Sigma_{N,j+1}
    =
    \left\{
        \begin{array}{cl}
            \rho_1 \Sigma_{N,j} + \gamma_1 w_{N,j} & {\rm for}\ j\ {\rm even}\\
            \rho_2 \Sigma_{N,j} + \gamma_2 w_{N,j} & {\rm for}\ j\ {\rm odd}\\
        \end{array}
    \right.,
\end{equation}
where the  initial element $\Sigma_{N,0}$ and the innovations $w_{N,0}, w_{N,1}, w_{N,2}, \ldots$ are independent standard normal random variables,
and use is made of (\ref{gamma12}). For any $N>0$, the conditional Gaussian distribution of $\Sigma_{N-1}$ given $\Sigma_N$ is degenerate since these random vectors have common entries:
\begin{equation}
\label{common}
    \Sigma_{N,2k}
    =
    \Sigma_{N-1,2k-1},
    \qquad
    1\< k <N.
\end{equation}
The other entries of $\Sigma_{N-1}$ (namely,  $\Sigma_{N-1,2k}$ for $0\< k <N$) are conditionally independent given $\Sigma_N$. More precisely,
\begin{equation}
\label{SS}
    \[[[
        (\Sigma_{N-1,2k})_{0\< k < N}
        \mid
        \Sigma_N
    \]]]
    =
    \mathop{\x}_{k=0}^{N-1}
    \[[[
        \Sigma_{N-1,2k}
        \mid
        \Sigma_{N,2k},\
        \Sigma_{N,2k+1},\
        \Sigma_{N,2k+2}
    \]]],
\end{equation}
where each of the conditional distributions on the right-hand side is described by the first of the Gaussian distributions from (\ref{cond_gauss}):
\begin{align}
\nonumber
    \[[[
        \Sigma_{N-1,2k}
        &\mid
        \Sigma_{N,2k},\
        \Sigma_{N,2k+1},\
        \Sigma_{N,2k+2}
    \]]]\\
\label{SSS}
    & =
    \cN(\rho_2\Sigma_{N,2k} - \rho_1\rho_2 \Sigma_{N,2k+1}+\rho_1 \Sigma_{N,2k+2},\ \gamma_1^2\gamma_2^2 ).
\end{align}

\begin{remark}
Notice the local nature of the conditional distribution $\[[[\Sigma_{N-1} \mid \Sigma_N\]]]$. Indeed, (\ref{common})--(\ref{SS}) imply that $\[[[\Sigma_{N-1,j} \mid \Sigma_N\]]] = \[[[\Sigma_{N-1,j} \mid \Sigma_{N,j}, \ \Sigma_{N,j+1}, \ \Sigma_{N,j+2}\]]]$ involves at most three entries of the random vector $\Sigma_N$ for every $0\< j \< 2N-2$. This property is inherited from the localness of the conditional probability measures describing the transitions of PRFs between the level sets.
\end{remark}

\section{Recurrence relation for conditional product moments}\label{sec:cond_prod_moms_GPRF}

Assuming $N\> 0$ and using the sets $\Gamma_0, \ldots, \Gamma_N$ from (\ref{Gamma}), we will now consider the two-dimensional Manhattan EDM system on the set
\begin{equation}
\label{Delta}
    \Delta_N
    :=
    \bigcup_{\ell=0}^{N}
    \Gamma_{\ell},
\end{equation}
which consists of $\#\Delta_N = N(N+5)/2+1$ lattice sites; see Fig.~\ref{fig:zigzag}. The product moment of the auxiliary GPRF $\xi$ over $\Delta_N$ is representable as
\begin{equation}
\label{MR}
    \bM_{\Delta_N}(\xi)
    =
    \bE R_N(\Sigma_N).
\end{equation}
Here, the function $R_N: \mR^{2N+1}\to \mR$ is defined as the conditional product moment of $\xi$ over the set $\Delta_N$ with respect to the random vector $\Sigma_N$ from (\ref{Sigma}):
\begin{equation}
\label{R}
    R_N(\Sigma_N)
    :=
    \bM_{\Delta_N}(\xi \mid \Sigma_N).
\end{equation}
In particular, since the set $\Delta_0$ is a singleton which consists of the origin of $\mZ^2$, the definition (\ref{R}) yields
\begin{equation}
\label{R0}
    R_0(s)
    =
    \bE(\xi_{00}\mid \sigma_{00} = s)
    =
    \mu+s.
\end{equation}

\begin{theorem}
\label{th:GPRF}
For any $N>0$, the conditional product moment function $R_N$ of the auxiliary GPRF $\xi$ for  the two-dimensional Manhattan EDM system satisfies the recurrence equation
\begin{eqnarray}
\nonumber
    R_N(\Sigma_N)
    & = &
    \bE
    (
            R_{N-1}(\Sigma_{N-1})
            \mid
        \Sigma_N
    )\\
\label{Rnext}
    &&\x (\mu + \Sigma_{N,0}) (\mu + \Sigma_{N,2N})
    \prod_{k=0}^{N-1} (\mu + \Sigma_{N,2k+1}),
\end{eqnarray}
with the initial condition (\ref{R0}). Here, the conditional expectation is taken over the conditional probability distribution $\[[[\Sigma_{N-1}\mid \Sigma_N\]]]$ described by (\ref{common})--(\ref{SSS}).
\end{theorem}
\begin{proof}
By partitioning the set $\Delta_N$ from (\ref{Delta}) into two disjoint subsets $\Delta_{N-1}$ and $\Delta_N\setminus \Delta_{N-1} = \Gamma_N \setminus \Delta_{N-1}$, with $\Gamma_N$ given by (\ref{Gamma}) (see also Fig.~\ref{fig:zigzag}), it follows that the function (\ref{R}) can be factorized as
\begin{eqnarray}
\nonumber
    R_N(\Sigma_N)
    & = &
    \bE
    \left(
    \left.
        \prod_{x \in \Gamma_N\setminus \Delta_{N-1}}
        \xi_x
        \prod_{y \in \Delta_{N-1}}
        \xi_y\,
        \right|\,
        \Sigma_N
    \right)\\
\label{Rsplit}
    & = &
    \bM_{\Delta_{N-1}}(\xi \mid \Sigma_N)
    \prod_{x \in \Gamma_N\setminus \Delta_{N-1}}
    \xi_x.
\end{eqnarray}
Here,
\begin{equation}
\label{GammaDelta}
    \prod_{x \in \Gamma_N\setminus \Delta_{N-1}}
    \xi_x
    =
    (\mu + \Sigma_{N,0}) (\mu + \Sigma_{N,2N})
    \prod_{k=0}^{N-1} (\mu + \Sigma_{N,2k+1})
\end{equation}
is a polynomial function of the appropriate entries of the random vector $\Sigma_N$ in (\ref{Sigma}). Furthermore,
\begin{eqnarray}
\nonumber
    \bM_{\Delta_{N-1}}(\xi \mid \Sigma_N)
    & = &
    \bE
    (\bM_{\Delta_{N-1}}(\xi \mid \Sigma_{N-1}, \Sigma_N)\mid \Sigma_N)\\
\nonumber
    & = &
    \bE
    (\bM_{\Delta_{N-1}}(\xi \mid \Sigma_{N-1})\mid \Sigma_N)    \\
\label{MMM}
    &=&
    \bE
    (R_{N-1}(\Sigma_{N-1})\mid \Sigma_N),
\end{eqnarray}
where, in addition to the tower property of iterated conditional expectations, we have used the relation $\[[[\xi_{\Delta_{N-1}} \mid \Sigma_{N-1}, \Sigma_N\]]] = \[[[\xi_{\Delta_{N-1}} \mid \Sigma_{N-1}\]]]$ which follows from the measurability of $\xi_{\Delta_{N-1}}$ with respect to $\Sigma_0, \ldots, \Sigma_{N-1}$ and the Markov property of the sequence $\Sigma_0, \ldots, \Sigma_N$. The recurrence equation (\ref{Rnext}) is now obtained by substitution of (\ref{GammaDelta}) and (\ref{MMM}) into (\ref{Rsplit}).
\end{proof}

By induction, (\ref{R0}) and (\ref{Rnext}) imply that the function $R_N$ is a polynomial (in $2N+1$ scalar variables) of degree $\#\Delta_N$. Therefore, calculating the partition function of the two-dimensional Manhattan EDM system for the set $\Delta_N$ through the product moment of the auxiliary GPRF $\xi$ according to (\ref{MR}) consists in evaluating the expectation of the recursively computed polynomial $R_N$ on the segment $\Sigma_N$ of the alternating autoregressive sequence governed by (\ref{AAR}). For example, by recalling the conditional distribution (\ref{cond_gauss})--(\ref{d}), it follows that
\begin{eqnarray*}
\nonumber
    R_1(s_0,s_1,s_2)
    & = &
    \bE(R_0(\sigma_{00}) \mid \sigma_{0,-1}=s_0,\, \sigma_{-1,-1}=s_1,\, \sigma_{-1,0}=s_2 )\\
\nonumber
    && \x (\mu+s_0)(\mu+s_1)(\mu+s_2)\\
\label{R1}
    & = &
        (\mu+\rho_2 s_0 - \rho_1\rho_2 s_1 + \rho_1 s_2)
        (\mu+s_0)(\mu+s_1)(\mu+s_2)
\end{eqnarray*}
is a quartic polynomial. Further iterations of the recurrence relation (\ref{Rnext}) involve the conditional expectations of higher order polynomials applied to conditionally independent Gaussian random variables; cf. (\ref{SS}), (\ref{SSS}). Although this recurrence is not easy to implement, its significance is that the above discussed localness  of the conditional probability distribution $\[[[\Sigma_{N-1} \mid \Sigma_N\]]]$, which specifies the linear operator $R_{N-1}\mapsto R_N$ in (\ref{Rnext}), appears to be a striking reduction of the long-range nature of interactions in the underlying Manhattan EDM system.\footnote{Note that similar recurrence relations, employed in the statistical mechanical transfer matrix method \cite{B_1982} and its corner version \cite{B_1981},  are known for  spin systems with \emph{short-range} (for example, nearest neighbour or round-a-face) interactions.}

\begin{remark}
The ``localization'' of the original problem above has been achieved  through its connection with the product moments of the auxiliary GPRF and the possibility of their recursive computation (in terms of the conditional product moments) due to the specific Markov structure of the random field. A similar line of reasoning, employing the aggregated level sets and associated conditional product moments, is applicable in three dimensions for the trivariate GPRFs from \cite{IG_1999}. This suggests that such an indirect approach to the monomer-dimer problem and its extensions may contain a hidden resource which is worth exploring.
\end{remark}

\section{One-dimensional Manhattan dimer-monomer system}\label{sec:1d_case}

Consider the Manhattan EDM system in the one-dimensional case, where the dimer potential (\ref{manpot}) reduces to
$
    W(z)
    :=
    \alpha |z|
$ for $z\in \mZ$,
with
$\alpha>0$ the attraction force parameter. In this case, the auxiliary GRF, which becomes a sequence $\xi:= (\xi_k)_{k \in \mZ}$, is a homogeneous Gaussian Markov chain  with mean $\mu$ from (\ref{mu}) and covariance function
$
    \cov(\xi_j,\xi_k)
    =
    \rho^{|j-k|}
$,
where, in accordance with (\ref{rhok}) and (\ref{mancov}),
\begin{equation}
\label{rho}
    \rho
    :=
    \re^{-\alpha\beta},
\end{equation}
and the shorthand notation (\ref{beta}) is used. In view of (\ref{MLE}), the MLE of the sequence $\xi$
takes the form
\begin{equation}
\label{muZ}
    \bm(\xi)
    =
    \lim_{N \to +\infty}
    \frac{\ln M_N}{N},
\end{equation}
where
\begin{equation}
\label{M_N}
    M_N
    :=
    \bM_{0:N-1}(\xi)
    =
    \bE
    \prod_{k=0}^{N-1}
    \xi_k
\end{equation}
is the product moment over the discrete interval $0:N-1$, with $a:b$ denoting the set $\mZ\bigcap [a,b]$   of integers from $a$ to $b$. For convenience of calculations, we will use a centered sequence $\sigma:= (\sigma_k)_{k \in \mZ}$, defined by
$
    \sigma_k := \xi_k - \mu
$
and satisfying an
 autoregressive equation
\begin{equation}
\label{SFDE}
    \sigma_{k+1}
    =
    \rho\sigma_k + \gamma w_k
\end{equation}
for $k\> 0$. Here, the initial condition $\sigma_0$ and the innovations $w_0, w_1, w_2,
\ldots$ are independent standard normal random variables, and
\begin{equation}
\label{gamma}
    \gamma
    :=
    \sqrt{1-\rho^2}.
\end{equation}
The autoregressive sequence $\sigma$ can be  obtained by
sampling an Ornstein-Uhlenbeck
diffusion process $Y := (Y_t)_{t \in \mR_+}$ with time step $\epsilon :=
\alpha \beta$ as
$
    \sigma_k
    :=
    Y_{k \epsilon}
$. The process $Y$ is governed by an
It\^{o} stochastic differential equation
$
    \rd Y_t = -Y_t\rd t + \sqrt{2} \rd \cW_t
$, where $\cW := (\cW_t)_{t \in \mR_+}$ is a standard Wiener process, independent of the  standard normal initial value $Y_0$.


\section{Conditional product moments and transfer operator}\label{sec:cond_prod_moms}

For the one-dimensional Manhattan EDM system of the previous section, the product moment (\ref{M_N})
of the auxiliary Gaussian Markov chain $\xi$
is representable as
\begin{equation}
\label{MN}
    M_N
    =
    \bE
    Q_N(\sigma_0).
\end{equation}
Here, for any $N \> 0$, the function $Q_N: \mR \to \mR$ is
defined by the conditional product moment
\begin{equation}
\label{Q}
    Q_N(\sigma_0)
    :=
    \bM_{0: N-1}(\xi\mid \sigma_0)
    =
    \bE
    \left(
        \prod_{k=0}^{N-1}
        \xi_k\,
        \Big|\,
        \sigma_0
    \right),
\end{equation}
with $Q_0 \equiv 1$. Note that $Q_N$ is a univariate counterpart of the function $R_N$ associated with the two-dimensional Manhattan EDM system by (\ref{R}). Thus, by employing a line of reasoning similar to Theorem~\ref{th:GPRF}, and  using the Markov property and homogeneity of the sequence $\xi$, it follows that the function $Q_{N+1}$ is expressed in terms of $Q_N$ from (\ref{Q})  through a linear \emph{transfer operator} $G$ as
\begin{eqnarray}
\nonumber
    Q_{N+1}(\sigma_0)
 &=&
    \bM_{0: N}(\xi\mid \sigma_0)\\
\nonumber
 & =&
    \xi_0\bE(\bM_{1:N}(\xi\mid \sigma_0, \sigma_1)\mid \sigma_0)
 =
    \xi_0 \bE(\bM_{1:N}(\xi\mid \sigma_1)\mid \sigma_0)\\
\label{nextQ}
 &=&
    (\mu + \sigma_0)
    \bE
    (
        Q_N(\sigma_1)
        \mid
        \sigma_0
    )
     =:
    G(Q_N)(\sigma_0).
\end{eqnarray}
Here, the last expectation is taken over the transition probability
law of the sequence $\sigma$ which, in view of (\ref{SFDE}) and
(\ref{gamma}), is described by
\begin{equation}
\label{1|0}
    \[[[\sigma_1 \mid \sigma_0\]]]
    =
    \cN(
        \rho \sigma_0,
        \gamma^2
    ),
\end{equation}
that is, the conditional Gaussian distribution
with mean $\bE(\sigma_1 \mid \sigma_0) =
\rho \sigma_0$, variance $\var(\sigma_1 \mid \sigma_0) = \gamma^2$ and PDF
\begin{equation}
\label{PDF_1|0}
    p_{1|0}(y\mid x)
    :=
    \re^{
        -
        (y-\rho x)^2/(2\gamma^2)}
    \big/ (\sqrt{2\pi}\gamma).
\end{equation}
 By induction, (\ref{nextQ}) and (\ref{1|0}) imply that the function $Q_N$ is a polynomial of degree
 $N$ with the leading coefficient
 $
    \lim_{x\to \infty}
    (Q_N(x)/x^N)
    =
    \prod_{k=0}^{N-1}
    \rho^k
    =
    \rho^{N(N-1)/2}
 $.
 The latter also employs the property that if $X \sim \cN(m,s^2)$ with a fixed variance $s^2$, and $L$ is a polynomial, then  $\bE L(X)$ is a polynomial in the mean value $m$ with the same leading term as $L$.  In particular, the first three of the polynomials $Q_N$ are computed as
\begin{equation*}
\begin{split}
Q_1(\sigma_0)
    & =
    \mu + \sigma_0,\\
Q_2(\sigma_0)
    & =
    (\mu + \sigma_0)
    \bE(\mu + \sigma_1\mid \sigma_0)
    =
    (\mu+\sigma_0)(\mu + \rho\sigma_0),\\
Q_3(\sigma_0)
    & =
    (\mu+\sigma_0)
    \bE((\mu+\sigma_1)(\mu + \rho\sigma_1)\mid \sigma_0)\\
    &=
    (\mu+\sigma_0)
    ((\mu+\rho \sigma_0)(\mu + \rho^2\sigma_0) + \rho(1-\rho^2)),
\end{split}
\end{equation*}
which shows that the recursive calculation of their coefficients, which are needed for (\ref{MN}),  in the standard basis of monomials $1, x, x^2, \ldots$ quickly becomes complicated. In Section~\ref{sec:matrix}, we will perform these calculations in a more suitable basis of Hermite polynomials after establishing the boundedness of the transfer operator $G$ which governs the recurrence equation  (\ref{nextQ}).

\begin{remark}
A similar recursive computation of conditional product moments can also be developed for a wider class of homogeneous Gaussian random sequences with rational spectral densities which are realizable as a hidden Markov chain with a higher dimensional state space.
\end{remark}

\section{An upper bound for the moment Lyapunov exponent.}\label{sec:upper}

Let $\mH$ denote the Hilbert space of functions $h: \mR \to \mR$ which are square integrable over the standard normal PDF $\nu$ from (\ref{nu}) and are endowed with the norm $\|h\|:= \sqrt{\bra h,h\ket} $ and inner product
$
    \bra f,g\ket
    :=
    \int_{\mR}
    f(x)g(x)
    \nu(x)
    \rd x
    =
    \bE (f(X)g(X))
$,
where $X$ is a standard normal random variable.
The Hermite polynomials (\ref{H}) form an orthogonal basis of the space $\mH$, with  $\|H_k\| = \sqrt{k!}$; see, for example, \cite[pp. 19--20]{J_1997}.
The following theorem establishes an upper bound for the  $\|\cdot\|$-induced norm of the
transfer operator $G$ defined by (\ref{nextQ}):
\begin{equation}
\label{normG}
    \sn G \sn
    :=
    \sup_{h \in \mH\setminus \{0\}}
    \frac{\|G(h)\|}{\|h\|}.
\end{equation}

\begin{theorem}
\label{th:upper}
The norm (\ref{normG}) of the transfer operator $G$ in (\ref{nextQ}), associated with the one-dimensional Manhattan EDM system, satisfies
\begin{equation}
\label{Gupper}
    \sn G\sn
     \<
    \frac{1}{\gamma}
    \sqrt{
        \mu^2
        +
        \frac
        {1+\rho^2}
        {\gamma^2}
    },
\end{equation}
\end{theorem}
\begin{proof}
Let $h \in \mH$. Then, by employing a change of measure
technique, it follows that for any real $x$,
\begin{eqnarray}
\nonumber
    G(h)(x)
     &=&
    (\mu+x)
    \int_{\mR}
    h(y)  p_{1|0}(y\mid x) \rd y\\
\label{Gq}
    & =&
    (\mu+x)
    \int_{\mR}
    h(y)
    q(x, y)
    \nu(y)\rd y=
    (\mu+x)
    \bra
        q(x,\cdot),
        h
    \ket.
\end{eqnarray}
Here,
\begin{equation}
\label{q}
    q(x, y)
    :=
    \frac
    {p_{1\mid 0}(y \mid x)}
    {\nu(y)}
    =
    \frac
    {p_{01}(x,y)}
    {\nu(x)\nu(y)},
\end{equation}
with $p_{1|0}$ the transition PDF of the sequence $\sigma$ described
by (\ref{PDF_1|0});\, $\nu$ is the standard normal PDF from (\ref{nu}),
and
\begin{equation}
\label{p01}
    p_{01}(x,y)
    :=
    \re^{
        (2\rho xy - x^2 - y^2)/(2\gamma^2)}
        \big/(2\pi \gamma)
\end{equation}
is the joint PDF of $\sigma_0$ and $\sigma_1$. Application of the
Cauchy-Bunyakovsky-Schwarz inequality to (\ref{Gq}) yields
\begin{equation}
\label{GG}
        (G(h)(x))^2
    \<
    (\mu+x)^2
    \|q(x,\cdot)\|^2
    \|h\|^2.
\end{equation}
An upper bound for the transfer operator norm
(\ref{normG}) can now be obtained by integrating both sides of (\ref{GG}) with respect to  the standard normal distribution and dividing the result by $\|h\|^2$:
\begin{eqnarray}
\nonumber
    \sn G \sn^2
    & \< &
    \int_{\mR}
    \nu(x)
    (\mu+x)^2
    \left(
         \int_{\mR}
             (q(x,y))^2
         \nu(y)
         \rd y
    \right)
    \rd x\\
\label{Gs}
    & = &
    \int_{\mR^2}
    (\mu+x)^2
    s(x, y)
    \rd x
    \rd y.
\end{eqnarray}
Here, in view of (\ref{q}), (\ref{p01}) and (\ref{gamma}),
\begin{eqnarray}
\nonumber
    s(x,y)
     &:= &
    (q(x, y))^2 \nu(x)\nu(y)=
    \frac
    {(p_{01}(x,y))^2}
    {\nu(x)\nu(y)}\\
\label{s}
     &= &
    \frac{1}{2\pi \gamma^2}
    \exp
    \left(
        \frac
        {4\rho xy-(1+\rho^2)x^2-(1+\rho^2)y^2}
        {2\gamma^2}
    \right).
\end{eqnarray}
The function $\gamma^2 s:\mR^2\to \mR_+$ is recognizable as a bivariate
Gaussian PDF with zero mean and covariance matrix
$$
    \left(
         \frac{1}{\gamma^2}
         \begin{bmatrix}
             1+\rho^2   & - 2\rho\\
             -2\rho     & 1 + \rho^2
         \end{bmatrix}
    \right)^{-1}
    =
    \frac{1}{\gamma^2}
    \begin{bmatrix}
        1+\rho^2    &  2\rho\\
        2\rho       & 1 + \rho^2
    \end{bmatrix},
$$
whose determinant is equal to one
and the diagonal entries describe the common marginal variance
$(1+\rho^2)/\gamma^2$.  It is the latter quantity that enters the right-hand side of (\ref{Gs}) through (\ref{s}) as
$$
    \sn G\sn^2
     \<
    \frac{1}{\gamma^2}
    \int_{\mR^2}
    \gamma^2 s(x, y) (\mu+x)^2 \rd x\rd y
    =
    \frac{1}{\gamma^2}
    \left(
        \mu^2
        +
        \frac
        {1+\rho^2}
        {\gamma^2}
    \right),
$$
thus proving (\ref{Gupper}).
\end{proof}

Now, by using the Cauchy-Bunyakovski-Schwarz inequality and the submultiplicativity of the operator norm $\sn \cdot\sn$, it follows that the product moment (\ref{MN}) satisfies
\begin{equation}
\label{MNbone}
    M_N
    =
    \bra
        Q_N,
        \bone
    \ket
    \<
    \|Q_N\|
    =
    \|G^N(\bone)\|
    \<
    \sn G\sn^N
\end{equation}
for any $N\>0$,
where $\bone$ denotes the identically unit function. Hence, the MLE (\ref{muZ}) admits an upper bound
$
    \bm(\xi)
    \<
    \ln \sn G\sn
$,
which, in combination with (\ref{Gupper}) and (\ref{gamma}), yields
\begin{equation}
\label{upper}
    \bm(\xi)
    \<
    \frac{1}{2}
    \ln
    \left(
        \mu^2
        +
        \frac{1+\rho^2}{1-\rho^2}
    \right)
    \underbrace{-
    \frac{1}{2}
    \ln(1-\rho^2)}_{\bI(\sigma_0, \sigma_1)}.
\end{equation}

\begin{remark}
The last term in (\ref{upper}) is recognizable as Shannon's mutual information $\bI(\sigma_0, \sigma_1)$ between  the random variables $\sigma_0$ and $\sigma_1$ (see \cite{CT_2006,G_2009}), which, in the Markov case being considered, coincides with $\bI(\sigma_1, \sigma_{\<0})$, where $\sigma_{\< 0}:= (\sigma_k)_{k \< 0}$ is the past history of the sequence $\sigma$ up until time 0.
\end{remark}

\begin{remark}
The proof of Theorem~\ref{th:upper} shows that the inequality
(\ref{Gs}) remains valid  for any homogeneous Markov
 random sequence $\xi$, and hence, its MLE, in this more general case, admits an upper bound
\begin{equation*}
\label{mupper}
    \bm(\xi)
    \<
    \frac{1}{2}
    \ln
    \int_{\mR^2}
    \frac{(p_{01}(x,y))^2}{\nu(x)\nu(y)}
    y^2
    \rd x\rd y,
\end{equation*}
where $\nu$ and $p_{01}$ are understood as the invariant PDF of
$\xi$ and the joint PDF of $\xi_0$ and $\xi_1$, respectively.
\end{remark}

\section{Matrix representation in the basis of Hermite polynomials}\label{sec:matrix}

As an orthonormal basis of the Hilbert space $\mH$ from Section~\ref{sec:upper}, we will use the functions $h_0, h_1, h_2, \ldots$  obtained by scaling the Hermite polynomials (\ref{H}) as
\begin{equation}
\label{h}
    h_k
    :=
    H_k/\sqrt{k!}.
\end{equation}
Their significance for computing the product moment (\ref{M_N}) is that, in view of the leftmost equality from (\ref{MNbone}),
\begin{equation}
\label{Zq}
    M_N
    =
    q_{N,0},
\end{equation}
where
\begin{equation}
\label{qNk}
    q_{N,k}
    :=
    \bra
        Q_N,
        h_k
    \ket
\end{equation}
are the Fourier coefficients of the polynomial $Q_N$ with respect to
the orthonormal basis (\ref{h}), so that
\begin{equation}
\label{QB}
    Q_N
    =
    \sum_{k=0}^{N}
    q_{N,k}
    h_k.
\end{equation}
We will show that the coefficients $q_{N, 0}, \ldots, q_{N, N}$
satisfy a recurrence equation in $N \> 0$ which can be ``encoded'' in  an equivalent
matrix representation of the transfer operator $G$ defined by
(\ref{nextQ}). To this end, we will employ a lemma below which, although it follows from the properties of the
Ornstein-Uhlenbeck operator \cite[Proposition~1.5.4(v) on
p.~233]{M_1995}, is provided with a proof for completeness of exposition.

\begin{lemma}
\label{lem:EHW}
Let $X$ be a Gaussian random variable with  mean $m$ and variance $s^2 < 1$. Then, for any nonnegative integer $k$,
\begin{equation}
\label{EHW}
    \bE H_k(X)
    =
    u^k
    H_k
    (
        m
        /u
    ),
    \qquad
    u:= \sqrt{1-s^2}.
\end{equation}
\end{lemma}
\begin{proof}
By combining the generating function (\ref{Hermgen}) of Hermite polynomials with the mo\-ment-generating function $\bE\re^{Xy }= \re^{m y+ s^2y^2/2}$ of the Gaussian
distribution $\cN(m,s^2)$, it follows that
$$
    \sum_{k = 0}^{+\infty}
        \frac
    {y^k}
    {k!}
    \bE
    H_k(X)
    =
    \re^{-y^2/2}
    \bE
        \re^{
            Xy
        }
     =
    \re^{
        m y - (1-s^2)y^2/2}
     =
    \sum_{k=0}^{+\infty}
    \frac{(uy)^k}{k!}
        H_k
    (
        m/u
    ),
$$
which holds for all $y$. Comparison of the leftmost and rightmost
sides of this identity yields (\ref{EHW}).
\end{proof}

Now, let $R$, $A$ and $A^{\dagger}$ be linear operators on the Hilbert space $\mH$ which  act on the basis functions
(\ref{h}) as
\begin{equation}
\label{RAA}
    R(h_k)
    :=
    \rho^k h_k,
    \qquad
    A(h_k)
    :=
    \sqrt{k}\,h_{k-1},
    \qquad
    A^{\dagger}(h_k)
    :=
    \sqrt{k+1}\,
    h_{k+1}.
\end{equation}
Both $A$ and $A^{\dagger}$ are defined (and are mutually adjoint) on the linear manifold of those functions $h\in \mH$ whose Fourier coefficients in the basis (\ref{h}) are square summable with an appropriate weight: $\sum_{k\>1}k \bra h, h_k\ket^2 <+\infty$.
The action of the operators $A$ and $A^{\dagger}$ in (\ref{RAA}) on the basis functions is identical to that of the
annihilation and creation operators on the eigenfunctions of
the Schr\"{o}dinger equation for the
quantum harmonic oscillator; see, for example,  \cite[pp. 52--53]{Meyer_1995} and \cite[p.~91]{S_1994}. In this sense,  $A+A^{\dagger}$ (which is a symmetric operator) corresponds to the quantum-mechanical position operator \cite[pp. 210--211]{J_1997}.
Although $A$ and $A^{\dagger}$ are unbounded,  their compositions $AR$ and $A^{\dagger}R$ with the operator $R$, as well as $R$ itself,  are bounded operators on the whole space $\mH$, since $0<\rho<1$ in view of (\ref{rho}).


\begin{theorem}
\label{th:GRAA} The transfer operator $G$ in (\ref{nextQ}) can be expressed in terms
of the operators $R$, $A$ and $A^{\dagger}$ from (\ref{RAA}), and its matrix representation in the
basis (\ref{h}) is given by
\begin{equation}
\label{GRAA}
    G
     =
     (\mu I
     +
     A + A^{\dagger})R
     =
    \begin{bmatrix}
    \mu     & \rho          & 0                 & 0                 & 0             & 0     & \ast\\
    1       & \mu\rho       & \sqrt{2}\rho^2    & 0                 & 0             & 0     & \ast\\
    0       & \sqrt{2}\rho  & \mu\rho^2         & \sqrt{3}\rho^3    & 0             & 0     & \ast\\
    0       &   0           & \sqrt{3}\rho^2    & \mu\rho^3         & \sqrt{4}\rho^4 & 0    & \ast\\
    0       & 0             &   0               & \sqrt{4}\rho^3    & \mu\rho^4     & \sqrt{5}\rho^5  & \ast\\
    0                   & 0 & 0   &   0   &    \sqrt{5}\rho^4    & \mu\rho^5      & \ast\\
    \ast                & \ast&\ast & \ast & \ast&\ast & \ast
    \end{bmatrix}.
\end{equation}
\end{theorem}
\begin{proof}
In view of (\ref{1|0}), application of Lemma~\ref{lem:EHW} with
$m := \rho\sigma_0$ and $s := \gamma$ from (\ref{gamma}) yields the identity
\begin{equation*}
\label{martingale}
\bE
    (
        H_k(\sigma_1)
        \mid
        \sigma_0
    )
    =
    \rho^k
    H_k(\sigma_0),
\end{equation*}
by which the random sequence
$(\rho^{-jk} H_k(\sigma_j))_{j\> 0}$ is a
martingale  with respect to the natural filtration of the homogeneous Gaussian Markov chain $\sigma$ for any given $k\> 0$.
Hence, the action of the transfer operator $G$ on the $k$th Hermite polynomial is
described by
\begin{eqnarray}
\nonumber
    G(H_k)(\sigma_0)
    & = &
    (\mu+\sigma_0)
    \bE
    (
        H_k(\sigma_1)
        \mid
        \sigma_0
    )
    =
    \rho^k
    (\mu+\sigma_0)
    H_k(\sigma_0)
\\
\label{GH}
    & =&
    \rho^k
    (
        \mu H_k(\sigma_0)
        +
        kH_{k-1}(\sigma_0)
        +
        H_{k+1}(\sigma_0)
    ),
\end{eqnarray}
where we have also used the recurrence relation $
    H_{k+1}(x)
    =
    xH_k(x)
    -k
    H_{k-1}(x)
$. Upon division of both parts of (\ref{GH}) by $\sqrt{k!}$, the result is representable in
terms of the basis functions (\ref{h}) and the operators (\ref{RAA}) as
\begin{eqnarray}
\nonumber
    G(h_k)
    & = &
    \rho^k
    (
        \mu h_k
        +
        \sqrt{k}\, h_{k-1}
        +
        \sqrt{k+1}\, h_{k+1}
    )\\
\label{GB}
    & = &
    \mu R(h_k)
    +
        A(R(h_k))
        +
        A^{\dagger}(R(h_k)).
\end{eqnarray}
Since the last relation holds for any $k \> 0$, the transfer operator $G$ is indeed expressed
in terms of $R$, $A$ and $A^{\dagger}$ as in (\ref{GRAA}). The first of the equalities (\ref{GB}) yields the $k$th column of the matrix representation of $G$ in the basis (\ref{h}) on the right-hand side of (\ref{GRAA}).
\end{proof}


Since $0< \rho<1$,  the matrix representation (\ref{GRAA}) implies that $G$ is a compact operator. With $\sqrt{R}$ denoting the square root of the operator $R$ from (\ref{RAA}) defined by $\sqrt{R}(h_k) = \rho^{k/2} h_k$,  the spectrum of $G$ coincides with that of a self-adjoint operator $\sqrt{R} (\mu I + A+ A^{\dagger}) \sqrt{R}$ and is all  real. Moreover,  since $\mu>0$,  the convex cone $\mH_+$  of functions $h\in \mH$ with all nonnegative Fourier coefficients  $\bra h,h_k\ket$ with respect to the orthonormal basis (\ref{h}) is invariant under $G$. Hence, by the Ruelle-Perron-Frobenius theorem, the operator $G$ has a unique (up to a positive scalar factor) eigenfunction
\begin{equation}
\label{Qeigen}
    Q(x)
    :=
    \sum_{k=0}^{+\infty}
    q_k h_k(x)
\end{equation}
in $\mH_+$,
which corresponds to the largest eigenvalue $\lambda$ of $G$.
By applying (\ref{GRAA}) to the Fourier coefficients (\ref{qNk}) of the polynomial $Q_N$, it follows that for any $N\> 0$,  they satisfy the recurrence equation
\begin{equation}
\label{qnext}
    q_{N+1,k}
    =
    \sqrt{k}\rho^{k-1}q_{N,k-1} + \mu \rho^k q_{N,k} + \sqrt{k+1} \rho^{k+1} q_{N,k+1},
    \qquad
    0\< k \< N+1.
\end{equation}
Here, $q_{N,k}=0$ for $k>N$ and $k<0$ in view of (\ref{QB}), and the recursion is initialized by $q_{0,0} = 1$ and $q_{0,k}=0$ for any $k\ne 0$. Since the cone $\mH_+$ is invariant under $G$, a straightforward induction yields the nonnegativeness $q_{N,k}\> 0$. Furthermore, the ratio of the free terms (\ref{Zq}) of the polynomials $Q_N$ and $Q_{N+1}$ converges to the largest eigenvalue of the transfer operator $G$ as
\begin{equation}
\label{rat}
    \lim_{N\to +\infty}
    \frac{q_{N+1,0}}{q_{N,0}}
    =
    \lambda,
\end{equation}
which can be used for finding this eigenvalue numerically  as it is done in the power method. The dependence of the MLE $\bm(\xi)=\ln \lambda$ on $\mu$ and $\rho$, obtained by numerical computation of the eigenvalue $\lambda$, is presented in Fig.~\ref{fig:mle_surf}.
\begin{figure}[htbp]
\vskip-3cm
\centering
    \includegraphics[width=11cm]{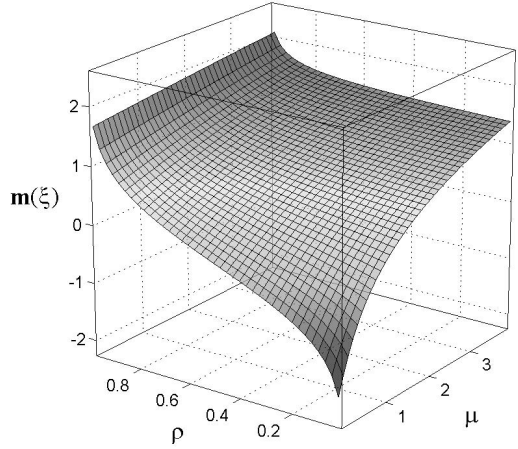}
\vskip-4cm
  \caption{
    The dependence of the MLE $\bm(\xi)$ of the auxiliary autoregressive Gaussian random sequence $\xi$ for the one-dimensional Manhattan EDM system on the parameters $\mu$ and $\rho$ from (\ref{mu}) and (\ref{rho}). Note that $\bm(\xi) \to \ln \mu$ as $\rho\to 0+$, and $\bm(\xi)\to +\infty$ as $\rho\to 1-$. Another limit,  $\lim_{\mu\to 0+}\bm(\xi)$, which is a function of $\rho$, is unknown.
    }
    \label{fig:mle_surf}
\end{figure}

\section{Connection with a pantograph equation eigenvalue problem}\label{sec:pantograph}

We will now discuss an alternative representation for the eigenvalue $\lambda$ from (\ref{rat}). Consider an exponential type generating function of the Fourier coefficients (\ref{qNk}):
\begin{eqnarray}
\nonumber
    T_N(t)
    &:= &
    \sum_{k=0}^{N}
    q_{N,k} \frac{t^k}{\sqrt{k!}}
    =
    \sum_{k\>0}
    \bra
        Q_N,h_k
    \ket
    \frac{t^k}{\sqrt{k!}}    \\
\label{TN}
    &=&
    \Bra
        Q_N,
        \sum_{k\>0}
        \frac{t^k}{k!} H_k(\cdot)
    \Ket
    =
    \re^{-t^2/2}
    \bE(Q_N(X)\re^{tX}),
\end{eqnarray}
where the second sum consists of finitely many nonzero terms (with $k\< N$), and $X$ is a standard normal random variable. Here, use has also been made of the generating function of Hermite polynomials (\ref{Hermgen}). The definition of $T_N$ implies that it is a polynomial of degree $N$.

\begin{theorem}
\label{the:Tnext}
The polynomials $T_N$, defined by (\ref{TN}), satisfy the recurrence equation
\begin{eqnarray}
\nonumber
    T_{N+1}(t)
    &=&
    (\mu+t) T_N(\rho t)
    +
    \rho T_N'(\rho t)\\
\label{Tnext}
    &=&
    \re^{-(\mu+t)^2/2}
    \d_t
    (
        T_N(\rho t)
        \re^{(\mu+t)^2/2}
    )
    =:
    K(T_N)(t)
\end{eqnarray}
\end{theorem}
\begin{proof}
Although the relation (\ref{Tnext}) follows from (\ref{qnext}), we will derive it in a more instructive fashion using a change of measure. By combining the definition (\ref{TN}) with the recurrence relation (\ref{nextQ}) and repeatedly using the tower property  of iterated conditional expectations, it follows that
\begin{align}
\nonumber
    T_{N+1}(t)
    & =
    \re^{-t^2/2}
    \bE(Q_{N+1}(\sigma_0)\re^{t\sigma_0})\\
\nonumber
    & =
    \re^{-t^2/2}
    \bE((\mu + \sigma_0)\re^{t\sigma_0}\bE(Q_N(\sigma_1)\mid \sigma_0))
    =
    \re^{-t^2/2}
    \bE((\mu + \sigma_0)\re^{t\sigma_0} Q_N(\sigma_1))\\
\label{Tnext1}
    & =
    \re^{-t^2/2}
    \bE(Q_N(\sigma_1)\bE((\mu + \sigma_0)\re^{t\sigma_0}\mid \sigma_1) ).
\end{align}
Since the Gaussian Markov chain $\sigma$ is reversible, the conditional probability law of $\sigma_0$, given $\sigma_1$, is obtained  from $\[[[\sigma_1\mid \sigma_0\]]]$ by swapping $\sigma_0$ and $\sigma_1$, so that
$
    \[[[\sigma_0\mid \sigma_1\]]]
    =
    \cN(\rho \sigma_1, \gamma^2)
$, with $\gamma$ given by (\ref{gamma}). Hence,
the conditional moment-generating function of $\sigma_0$ is
\begin{equation}
\label{theta}
    \theta_t(\sigma_1)
    :=
    \bE(\re^{t\sigma_0} \mid \sigma_1)
    =
    \re^{t\rho\sigma_1 + t^2 \gamma^2/2}.
\end{equation}
Now, since
\begin{equation}
\label{thetader}
    \bE((\mu + \sigma_0)\re^{t\sigma_0}\mid \sigma_1)
    =
    (\mu +\d_t)\theta_t(\sigma_1),
\end{equation}
then substitution of (\ref{theta}) and (\ref{thetader}) into the right-hand side of (\ref{Tnext1}) represents the latter equation as
\begin{eqnarray}
\nonumber
    T_{N+1}(t)
    & =&
    \re^{-t^2/2}
    (\mu +\d_t)
    \bE(Q_N(\sigma_1)\theta_t(\sigma_1))\\
\nonumber
    & =&
    \re^{-t^2/2}
    (\mu +\d_t)
    (
        \re^{t^2/2}
        \re^{-(\rho t)^2/2}
        \bE(Q_N(\sigma_1)\re^{t\rho \sigma_1})
    )\\
\label{Tnext2}
    & =&
    \re^{-t^2/2}
    (\mu +\d_t)
    (
        \re^{t^2/2}
        T_N(\rho t)
    )
    =
    (\mu+t) T_N(\rho t)
    +
    \rho T_N'(\rho t),
\end{eqnarray}
thus establishing (\ref{Tnext}). In (\ref{Tnext2}), use has again been made of  (\ref{gamma}) and (\ref{TN}).
\end{proof}

The linear operator $K$, defined by (\ref{Tnext}), is the composition of a linear first order  differential operator with an argument scaling  operator. Since $K$ consists in  differentiation and scaling,
it is easier to iterate than the integral operator $G$ in (\ref{nextQ}) which involves conditional expectation. With the initial condition $T_0\equiv 1$,  the subsequent three of the functions $T_N$ are computed as
\begin{eqnarray*}
    T_1(t) & =& \mu + t,\\
    T_2(t) & =& (\mu +t) (\mu+\rho t) + \rho,\\ 
    T_3(t) & =& (\mu +t) ((\mu +\rho t) (\mu+\rho^2 t) + \rho) + \rho (\mu+\rho^2 t + \rho(\mu+\rho t)).
\end{eqnarray*}
By the relation $T_{N+1}(0) = \mu T_N(0) + \rho T_N'(0)$, which follows from (\ref{Tnext}), the limit (\ref{rat}) is representable as
$
    \lambda
    =
    \lim_{N\to +\infty}
    (T_{N+1}(0)/T_N(0))
    =
    \mu + \rho \lim_{N\to +\infty}(T_N'(0)/T_N(0))
$.
Here, the rightmost ratio, which is the logarithmic derivative $(\ln T_N(t))'|_{t=0}$, is equal to $q_{N,1}/q_{N,0}$.
The generating function of the Fourier coefficients of the eigenfunction (\ref{Qeigen}), defined by
\begin{equation}
\label{T}
    T(t)
    :=
    \sum_{k=0}^{+\infty}
    q_k \frac{t^k}{\sqrt{k!}},
\end{equation}
is a solution of the eigenvalue problem  $K(T) = \lambda T$  for the  functional-differential operator $K$, that is,
$
    (\mu+t) T(\rho t)
    +
    \rho T'(\rho t)
    =
    \lambda T(t)
$.
Upon introducing a rescaled independent variable   $\tau:= \rho t$, the eigenvalue problem takes the form
\begin{equation}
\label{pant}
    T'(\tau)
    =
    \frac{\lambda}{\rho} T(\tau/\rho)
    -\frac{\mu+\tau/\rho}{\rho}
    T(\tau),
\end{equation}
which is a pantograph equation \cite{BDMO_2007,D_1990,KM_1971,M_1940,S_1995} with a variable coefficient. The presence of the independent variable $\tau$ together with its scaling $\tau/\rho$ as arguments of the unknown function $T$ makes this functional-differential equation hard to solve, let alone finding $\lambda$ as the largest eigenvalue for which the pantograph equation (\ref{pant}) has a feasible solution (\ref{T}).
\section{Concluding remarks}\label{sec:conclusion}

We  have established a novel connection between the product moments and moment Lyapunov exponents of homogeneous Gaussian random fields on a multidimensional integer lattice with the partition function of an elastic dimer-monomer system which extends the classical monomer-dimer problem of equilibrium statistical mechanics by allowing for long-range interactions. A salient feature of the  auxiliary GRF,  which encodes the thermodynamic properties of the EDM system, is that  its mean value and the covariance function are the Boltzmann factors, associated  with  the energetics of the underlying particle system, and,  therefore, are  both nonnegative. Any insight into how the MLE of such a GRF depends on its spectral density function in an arbitrary dimension would shed light on the solution of the EDM problem, including its particular case, the classical monomer-dimer problem,  which remains unsolved in dimensions three and higher.

To this end, we have outlined a research paradigm which builds on the observation that a specific spatial Markov structure of the GRF, present, for example,  in Gaussian Pickard random fields,  facilitates recursive representation of the product moments of such GRFs. We have discussed the possibility to recursively compute the product moments with a local  transition from one level set to another,  for a Gaussian PRF, which is the auxiliary GRF for the Manhattan EDM system.
This  paradoxical ``localization'' of the effect of long-range interactions  suggests that the product moment approach, proposed in the present study, is a promising resource towards the solution of the monomer-dimer problem and its generalizations.

In addition to its links with the statistical mechanical lattice models of interacting particle systems, the problem of computing the product moments and MLEs for homogeneous GRFs
appears to be of interest in its own right from the probability theoretic point of view.  In this regard, even the deceptively simple  case of univariate Gaussian Markov chains reveals rich algebraic and probabilistic structures. In particular, we have linked  the product moments and MLEs in this case to the ladder operators of the quantum harmonic oscillator and the pantograph functional-differential equation by using a change of measure technique.

\section*{Acknowledgments}

The author  thanks Professor Nikolai N.
Leonenko for helpful discussions of the results of the paper. Support of the Australian Research Council is also gratefully acknowledged.


\end{document}